\documentclass[a4paper,11pt]{amsart}

\usepackage{fullpage}%
\usepackage{graphicx}%
\usepackage{multirow}%
\usepackage{amsmath,amssymb,amsfonts}%
\usepackage{amsthm}%
\usepackage[title]{appendix}%
\usepackage{xcolor}%
\usepackage{textcomp}%
\usepackage{manyfoot}%
\usepackage{booktabs}%
\usepackage{algorithm}%
\usepackage{algorithmicx}%
\usepackage{algpseudocode}%
\usepackage{listings}%
\usepackage{framed}
\usepackage{nameref}
\usepackage{orcidlink}
\usepackage{mathrsfs}
\usepackage[shortlabels]{enumitem}
\usepackage{mathtools}
\usepackage{bm}
\usepackage{relsize}

\usepackage{esint}
\usepackage{hyperref}
\usepackage{cleveref}
\usepackage[abbrev]{amsrefs}

\theoremstyle{plain}
\newtheorem{thm}{Theorem}[section]

\newtheorem{lemma}[thm]{Lemma}
\newtheorem{prop}[thm]{Proposition}

\newtheorem{cor}[thm]{Corollary}
\newtheorem{corollary}[thm]{Corollary}
\newtheorem{con}[thm]{Conjecture}
\theoremstyle{definition}

\newtheorem{rem}[thm]{Remark}

\DeclareMathOperator{\inter}{\textup{int}}

\DeclareMathOperator{\relint}{\textup{relint}}

\DeclareMathOperator{\supp}{supp}
\DeclareMathOperator{\conv}{conv}

\newcommand{\calE}{\mathcal E}
\newcommand{\calF}{\mathcal F}
\newcommand{\calI}{\mathcal{I}}
\newcommand{\calJ}{\mathcal{J}}
\newcommand{\calK}{\mathcal{K}}
\newcommand{\calL}{\mathcal L}
\newcommand{\R}{\mathbb{R}}
\newcommand{\RR}{\mathbb{R}}

\newcommand{\NN}{\mathbb{N}}
\newcommand{\ZZ}{\mathbb Z}

\newcommand{\ba}{\bm a}
\newcommand{\bj}{\bm j}
\newcommand{\bk}{\bm k}
\newcommand{\bu}{\bm u}
\newcommand{\bt}{\bm t}
\newcommand{\bs}{\bm s}
\newcommand{\bx}{\bm x}
\newcommand{\by}{\bm y}
\newcommand{\bz}{\bm z}
\newcommand{\bw}{\bm w}
\renewcommand{\d}{\,\textup d}

\renewcommand{\mid}{\,:\,}

\newenvironment{enuma}{
	\begin{enumerate}[label=\textnormal{(\alph*)},leftmargin=*,align=left,itemindent=0pt,labelwidth=0pt,labelsep=0pt, listparindent=\parindent,parsep=\parskip,topsep=2pt,itemsep=2pt,partopsep=4pt]}
{	
	\end{enumerate}
}

\allowdisplaybreaks

\BibSpec{article}{%
  +{}{\PrintAuthors}  		{author}
  +{,}{ \textit}     		{title}
  +{,}{ }             		{journal}
  +{}{ \textbf}       		{volume}
  +{}{ \parenthesize} 		{date}
  +{}{, no. } 		{number}
  +{,}{ }      	      		{conference}
  +{,}{ }      	      		{book}
  +{,}{ }            		{pages}
  +{,}{ }            	 	{note}
  +{,}{ }            	 	{status}
  +{,}{  \texttt } {eprint}
  +{.}{}              {transition}
}

\BibSpec{book}{%
  +{}{\PrintAuthors}  {author}
  +{,}{ \textit}      {title}
  +{,}{ }	      {note}
  +{,}{ }             {publisher}
  +{,}{ }             {place}
  +{,}{ }             {date}
  +{.}{}              {transition}
}

\BibSpec{incollection}{%
    +{}  {\PrintAuthors}                {author}
    +{,} { \textit}                     {title}
    +{.} { }                            {part}
    +{:} { \textit}                     {subtitle}
    +{,} { \PrintContributions}         {contribution}
    +{,} { \PrintConference}            {conference}
    +{,} { }                            {booktitle}
    +{,} { pp.~}                        {pages}
    +{,}{ }       		{series}
    +{}{ \textbf}       		{volume}
    +{,} { }                            {publisher}
    +{,} { }                 {date}
    +{,} { }                            {status}
    +{,} { \PrintDOI}                   {doi}
    +{,} { available at \eprint}        {eprint}
    +{}  { \parenthesize}               {language}
    +{}  { \PrintTranslation}           {translation}
    +{;} { \PrintReprint}               {reprint}
    +{.} { }                            {note}
    +{.} {}                             {transition}
}

\begin{document}

\author{Filip Fryš}
\author{Jan Kotrbat{\'y}}
\thanks{FF was supported by Charles University grants PRIMUS/24/SCI/009 and GAUK/34126. JK was supported by Charles University grants  PRIMUS/24/SCI/009 and UNCE/24/SCI/022.}
 
\address{Charles University, Faculty of Mathematics and Physics, Mathematical Institute of Charles University, Sokolovsk\'a 49/83, 186 00 Prague, Czechia}
\email{filip.frys@matfyz.cuni.cz}
\email{kotrbaty@karlin.mff.cuni.cz}
\date{\today}

\title[]{Around higher-order Godbersen conjectures}

\begin{abstract}
We consider two higher-order generalizations of the Godbersen conjecture for mixed volumes of convex bodies, which were first proposed by Schneider in 2000. First, we establish the conjectures in several special cases, in particular in low dimensions. Second, we prove that certain consequences of the conjectures hold. More precisely, we define a higher-order version of the unbalanced difference body and prove related weighted inequalities, generalizing previous results of Artstein-Avidan and Putterman. Finally, we introduce higher-order analogs of unbalanced joins of convex bodies considered by Artstein-Avidan, Einhorn, Florentin, and Ostrover, prove the corresponding volume bounds, and propose conjectures that interpolate between the higher-order Godbersen conjectures and a conjecture due to F\'ary and R\'edei.
\end{abstract}

\maketitle

%\tableofcontents

\section{Introduction}\label{secIntro}

The Godbersen conjecture asserts that the geometric invariants associated with a convex body in Euclidean space via mixed volumes with the opposite body are maximized by  simplices. Dating back to the 1938 thesis of C. Godbersen \cite{Godbersen}, it has been one of the longest-standing open problems in convex geometry, and until very recently, only special cases and consequences of the conjectured inequality have been known.

A notable example of such a consequence is the classical Rogers--Shephard inequality \cites{RogersShephard57,RogersShephard58}. It states that, among all convex bodies $K\subset\RR^n$ of fixed positive volume $|K|$, the volume of the corresponding difference body $DK=\{k-l\mid k,l\in K\}$ is maximal for simplices; more precisely
\begin{align}
\label{eq:RS}
|DK|\leq {2n\choose n}|K|,
\end{align}
where equality holds if and only if $K$ is a simplex.

It is a remarkable fact and a cornerstone of the Brunn--Minkowski theory that the Lebesgue measure on $\RR^n$ polarizes when restricted to the class $\calK(\RR^n)$ of convex bodies. In other words, there exists a (unique) function $V:\calK(\RR^n)^n\to\RR$, called the mixed volume, that is symmetric, multilinear and satisfies $V(K,\dots,K)=|K|$. Here linearity is with respect to the Minkowski addition $K+L=\{k+l\mid k\in K,l\in L\}$ and scaling $r K=\{rk\mid k\in K\}$ by a \emph{non-negative} factor $r\geq0$. Let us emphasize that although one can scale $K$ by any $r\in\RR$, the mixed volume is not homogeneous for $r<0$.

Observe that $DK=-K+K$ is the Minkowski sum of $K$ with its reflection $-K=(-1)K$ about the origin. Expanding the left-hand side of \eqref{eq:RS} by multilinearity and the right-hand side using ${2n\choose n}=\sum_{j=0}^n {n\choose j}^2$, Godbersen \cite{Godbersen} and later independently Makai Jr. \cite{Makai} conjectured that the inequality holds terms by term:

\begin{con}[Godbersen conjecture]\label{Godbc}
Let $0\leq j\leq n$. For any $K\in\mathcal{K}(\mathbb{R}^n)$ one has
\begin{equation}\label{Godb}
V\big(-K[j], K[n-j]\big)\leq {n\choose j}|K|.
\end{equation}
Moreover, if  $0<j<n$ and $\dim K = n$, equality holds if and only if $K$ is a simplex.
\end{con}

The conjecture holds trivially for $j\in\{0,n\}$ and follows easily for $j\in\{1,n-1\}$ from the properties of the Minkowski asymmetry measure, see \S\ref{ss:asymmetry}. In particular, it is true in dimensions $n\leq3$. Godbersen \cite{Godbersen} further verified \eqref{Godb} in the case when $K$ is of constant width, i.e., when $DK$ is a Euclidean ball. Much more recently, Conjecture \ref{Godbc} was verified for anti-blocking convex bodies by Artstein-Avidan, Sadovsky, and Sanyal \cite{ArtsteinSadovskySanyal23}, and Sadovsky \cite{Sadovsky25} subsequently extended this result to the class of locally anti-blocking bodies, see \cites{ArtsteinSadovskySanyal23,Sadovsky25} for precise definitions of these classes. Very recently, the second-named author and Mouamine \cite{KM}, using a surprisingly simple argument based on the monotonicity of mixed volume, proved \eqref{Godb} for general convex bodies and established the characterization of equality cases for convex polytopes.

The proof of \eqref{Godb} given in \cite{KM} is directly motivated by the renewed attention recently attracted by the so-called higher-order generalization of the Rogers--Shephard and related inequalities, see, e.g., \cites{HaddadETAL23, LangharstETAL:mOrder,  LangharstSolaUlivelli24, HaddadETAL25, LangharstRoysdonZhao25,LangharstXi24, HaddadETAL25+,Langharst:Comments,ZhouETAL25,KW22} and reference therein. This concept goes back to Schneider \cite{Schneider70} who considered for general $p\in\NN$ the higher-order difference body $D_pK=-\Delta_pK+K^p$, where $\Delta_p:\RR^n\to\RR^{pn}$ is the diagonal embedding, and proved the following generalization of \eqref{eq:RS}:
    \begin{align}
    \label{eq:Schneider}
    |D_pK|\leq {np+n \choose n}|K|^p 
    \end{align}
    with equality if and only if $K$ is a simplex (unless $\dim(K)<n$). Note that, strictly speaking, Schneider \cite{Schneider70} used a different sign convention for $D_pK$, see \S\ref{ss:higherDK} for more details.

Expanding the left-hand side of \eqref{eq:Schneider} in terms of mixed volumes and using an appropriate combinatorial identity  for ${np+n \choose n}$, one may again conjecture that the inequality holds term by term. This was done recently by the second-named author \cite{K25}.

\begin{con}[{\cite[Conjecture 1.2]{K25}}]\label{KK1}
Let  $p\in\NN$ and $0\leq j\leq n$. For any  $K\in\mathcal{K}(\mathbb{R}^n)$ one has
\begin{equation}\label{KOTR}
    V\big(-\Delta_pK[j],K^p[np-j]\big)\leq{n\choose j}|K|^p.
\end{equation}
Moreover, if $p\geq2$, $j>0$, and $\dim(K)=n$, equality holds if and only if $K$ is a simplex.
\end{con}

Conjecture \ref{KK1}  can be further refined using $K^p=\iota_1K+\dots+\iota_pK$, where $\iota_i:\RR^n\to\RR^{pn}$ is the inclusion into the $i$-th factor of $\RR^{pn}=\RR^n\times\cdots\RR^n$. In this way one arrives at

\begin{con}[{\cite[Conjecture 1.3]{K25}}]\label{KK2}
Let $p\in\NN$ and $j_0,\dots,j_p\in\NN_0$ such that $j_0+\cdots+j_p=n$. For any $K\in\calK(\RR^n)$ one has
\begin{align}
\label{KOTRB}
    V\big(-\Delta_pK[n-j_0],\iota_1K[n-j_1],\dots,\iota_pK[n-j_p]\big)\leq\frac{(n!)^p}{(np)!}{n\choose j_0, \dots, j_p}|K|^p.
    \end{align}
Moreover, if $p\geq2$, $j_0, \dots, j_p<n$, and $\dim(K)=n$, equality holds if and only if $K$ is a simplex.
\end{con}

As pointed out to us by E. Putterman, Conjecture \ref{KK2} was in fact already made in 2000 by Schneider, see \cite[p. 537]{Schneider00}. Schneider also explains that the problem dates back to the seminal paper of Janson \cite{Janson} and attributes to Janson the proof of \eqref{KOTRB} in the special case  $p=n-1$ and $j_0=\dots=j_{n-1}=1$, see \cite{Schneider00} and Proposition \ref{janson} below. Recently, the second-named author \cite[Theorem 1.4]{K25}  proved that Conjecture \ref{KK2} (and hence Conjecture \ref{KK1}) holds for the class of anti-blocking convex bodies.

\subsection{Our results}

In this paper,  motivated by the previous developments in the classical case $p=1$, we provide additional evidence for validity of the higher-order Godbersen conjectures. More precisely, we verify Conjectures \ref{KK1} and \ref{KK2} in some special cases and prove that certain potential consequences of the conjectures are true; moreover, we also formulate new conjectures, both weaker and stronger than Conjectures \ref{KK1} and \ref{KK2}.

\subsubsection{Special cases of the higher-order Godbersen conjectures}

First, we prove Conjectures \ref{KK1} and \ref{KK2} for some specific values of the parameters $j$ and $j_0,\dots,j_p$, respectively. These results in particular imply that Conjecture \ref{KK1} is true in dimensions $n\leq 3$ and Conjecture \ref{KK2} for $n\leq2$. Moreover, using Janson's inequality mentioned above, the latter extends to dimension three, however without the characterization of equality cases.

Second, using a generalization of an argument due to Godbersen \cite{Godbersen}, we verify Conjectures \ref{KK1} and \ref{KK2} for low-dimensional convex bodies of constant width. Notice that, unlike in the case $p=1$, \eqref{KOTR} is not symmetric in $j$ and $n-j$. Consequently, the argument works in general dimension only if $j\leq\frac n2$ in Conjecture \ref{KK1}. Along the way, we also observe that both conjectures are true for symmetric convex bodies (of any dimension). Although this is not difficult to see, it is not as obvious as for $p=1$.

Third, we show that the higher-order difference body of a product body is the product of higher-order difference bodies and conclude that the validity of both conjectures is stable under the operation of taking cartesian products.

To summarize, our first main result is

\begin{thm}
\label{thm:main1}
The following special cases of Conjectures  \ref{KK1} and  \ref{KK2} are true:
\begin{enuma}
\item  Conjecture \ref{KK1} holds for $n\leq 3$.
\item Conjecture \ref{KK2} holds for $n\leq2$. If $n=3$, then \eqref{KOTRB} holds for any $K\in\calK(\RR^3)$.
\item Conjecture \ref{KK1} holds for convex bodies of constant width if $n\leq6$ or if $j\leq \lfloor\frac n2\rfloor$.
\item Conjecture \ref{KK2} holds for convex bodies of constant width if $n\leq3$. If $n\in\{4,5,6\}$, then \eqref{KOTRB} holds for any $K\in\calK(\RR^n)$ of constant width.
\item Conjecture \ref{KK2} (and hence Conjecture \ref{KK1}) holds for symmetric convex bodies.
\item Fix $p\in\mathbb{N}$. If inequality \eqref{KOTRB} holds for $K_i\in\calK(\RR^{n_i})$, $i=1,2$, then the cartesian product $K_1\times K_2\in\calK(\RR^{n_1+n_2})$ satisfies Conjecture \ref{KK2}, and analogously for Conjecture \ref{KK1}.  
\end{enuma}
\end{thm}

Let us point out that item (f) of Theorem \ref{thm:main1} can be iterated and combined with other items as well as with the previously established case of anti-blocking convex bodies \cite{K25}. In this way we obtain the validity of Conjectures \ref{KK1} and \ref{KK2} for finite cartesian products of bodies from the corresponding classes. 

\subsubsection{Higher-order unbalanced difference bodies and weighted inequalities}

\label{ss:intro:weighted}

As the Godbersen conjecture had long appeared out of reach, various consequences of it were studied as potential evidence in its favour. In a recent paper, Artstein-Avidan and Putterman \cite{ArtsteinPutterman25} proved two different inequalities for weighted sums of mixed volumes $V\big(-K[j], K[n-j]\big)$, that would follow at once from \eqref{Godb}. One of them in particular implies that the Godbersen conjecture holds  in a certain sense on average, cf. Corollary \ref{cor:average} below. Moreover, Artstein-Avidan and Putterman defined for a convex body $K\in\calK(\RR^n)$ and general $t\in[0,1]$ the unbalanced difference body
$$
D^tK:=-tK+(1-t)K
$$
and conjectured the corresponding Rogers--Shephard-type inequality. They also showed that this conjecture, weaker than Conjecture \ref{Godbc} and subsuming \eqref{eq:RS} for $t=\frac12$, would imply both of their weighted inequalities.

Here we generalize the inequalities of Artstein-Avidan and Putterman \cite{ArtsteinPutterman25} to the higher-order setting. We emphasize that although we list below for simplicity only the corresponding results related to Conjecture \ref{KK1}, we also obtain a precisely analogous set of statements related to the more general Conjecture \ref{KK2}. Our first result in this direction and the second main result of the paper is a direct generalization of \cite[Theorem 4]{ArtsteinPutterman25}:

\begin{thm}\label{T4}
   Let $p\in\mathbb{N}$. For any $K\in\mathcal{K}(\mathbb{R}^n)$ and $t\in [0, 1]$, one has
   \begin{equation}\label{Eq71}
        \sum_{j=0}^n t^j(1-t)^{n-j}V( -\Delta_pK[j],K^p[np-j])\leq|K|^p.
    \end{equation}
    If $\dim(K)=n$ and equality holds for some $t\in (0, 1)$, then $K$ is a simplex.
\end{thm}

Clearly, Theorem \ref{T4} would follow from Conjecture \ref{KK1} if the latter is true. Furthermore, integrating \eqref{Eq71} with respect to $t$ over the interval $[0, 1]$ yields at once the following inequality, which generalizes \cite[Corollary 5]{ArtsteinPutterman25} and shows that \eqref{KOTR} holds ``on average'':
\begin{cor}
\label{cor:average}
   Let $p\in\mathbb{N}$. For any $K\in\mathcal{K}(\mathbb{R}^n)$, one has
    $$\frac{1}{n+1}\sum_{j=0}^n{n \choose j}^{-1}V\big( -\Delta_pK[j],K^p[np-j]\big)\leq|K|^p.$$
\end{cor}

Our third main result is a weighted inequality for the mixed volume of $-\Delta_pK$ and $K^p$ with different weights than in Theorem \ref{T4}, generalizing \cite[Theorem 7]{ArtsteinPutterman25}. Note that a straightforward computation shows that also the following theorem would be implied by Conjecture \ref{KK1}:

\begin{thm}\label{T5}
Let $p\in\mathbb{N}$. For any $K\in\mathcal{K}(\mathbb{R}^n)$ and $t\in [0, 1]$, one has
\begin{align}\label{f00}
    \sum_{j=0}^n c_{n, p,  j}(t)V\big(-\Delta_pK[j],K^p[np-j]\big) \leq |K|^p, 
\end{align}
where
$$
c_{n, p, j}(t)\coloneqq\sum_{m=0}^{np-j}{np-m \choose j}\frac{(np)!n!}{m!(np+n-m)!}(1-t)^mt^{np-m}.
$$
If $\dim(K)=n$ and equality in \eqref{f00} holds for some $t\in (0, 1)$, then $K$ is a simplex.
\end{thm}

Finally, for any $p\in\mathbb{N}$ and $t\in[0, 1]$ we define the higher-order unbalanced difference body corresponding to a convex body $K\in\mathcal{K}(\mathbb{R}^n)$ as 
$$D_p^tK\coloneqq (1-t)K^p-t\Delta_pK.$$
Notice that for $p=1$ we recover the unbalanced difference body of Artstein-Avidan and Putterman, i.e., $D_1^tK=D^tK$. In this connection, motivated by \cite[Conjecture 2]{ArtsteinPutterman25}, we propose
\begin{con}\label{COn}
   Let $p\in\mathbb{N}$. For any $K\in\mathcal{K}(\mathbb{R}^n)$ with $\dim(K)=n$ and any $t\in[0, 1]$, one has
    \begin{align}
    \label{EQ47ll}
            \frac{|D_p^tK|}{|K|^p}\leq\frac{|D_p^tS_n|}{|S_n|^p},
        \end{align}
    where $S_n\in\calK(\RR^n)$ is an $n$-simplex. If equality holds for some $t\in(0, 1)$, then $K$ is a simplex.
\end{con}
    Observe that Conjecture \ref{COn} is trivial for $t=0$ and $t=1$,  and holds for $t=\frac{1}{2}$ by Schneider's inequality \eqref{eq:Schneider}.  Expanding both sides of \eqref{EQ47ll} using multilinearity of mixed volumes, it is immediate that Conjecture \ref{KK1} implies Conjecture \ref{COn}. Moreover, we prove that even the weaker Conjecture \ref{COn} implies Theorems \ref{T4} and \ref{T5}, extending the same phenomenon from the case $p=1$ considered in \cite{ArtsteinPutterman25} to general $p\in\NN$, see Propositions \ref{pro:impliesT4} and \ref{pro:impliesT5}, respectively.

As pointed out above, we also prove analogous results for the mixed volume of $-\Delta_pK$ and $\iota_1K,\dots,\iota_pK$, that is, related to Conjecture \ref{KK2}. More precisely, the analogs of Theorem \ref{T4} and Corollary \ref{cor:average} are Theorem \ref{T4a} and Corolary \ref{cor:average_mixed}, respectively. Theorem \ref{T5} is in fact proven as a special case of its analog in this more general setting, Theorem \ref{T5a}. Finally, we also define for each $p\in\NN$, $K\in\calK(\RR^n)$, and $\bm t=(t_0,\dots,t_p)\in \RR_+^{p+1}$ satisfying $\sum_{l=0}^pt_l=1$, the more general version of unbalanced higher-order difference body
$$
D_p^{\bt}K:=-t_0\Delta_pK+\sum_{l=1}^pt_l\iota_lK,
$$
and propose Conjecture \ref{COn2} analogous to Conjecture \ref{COn}. To complete the analogy, we show in Propositions \ref{pro:impliesT4a} and \ref{pro:impliesT5a} that Conjecture \ref{COn2} implies Theorems \ref{T4a} and \ref{T5a}, respectively. Note that also all these results boil down to the corresponding results of \cite{ArtsteinPutterman25} for $p=1$.

\subsubsection{Inequalities for higher-order unbalanced joins of convex bodies}
\label{ss:intro:join}

Finally, we generalize to the higher-order setting the results of  Artstein-Avidan, Einhorn, Florentin, and Ostrover \cite{ArtsteinETAL15} who considered the unbalanced join of $K$ and $-K$. More precisely, they proved a sharp upper bound on the volume of $-tK\vee(1-t)K=\conv\big(-tK\cup(1-t)K\big)$ and proposed a related conjecture that interpolates between Conjecture \ref{Godbc} and a long-standing conjecture due to F\'ary and R\'edei \cite{FaryRedei}.

Our main result in this direction is the following inequality that generalizes \cite[Theorem 1.5]{ArtsteinETAL15}:
\begin{thm}
\label{thm:join}
Let $p\in\mathbb{N}$. For any $K\in\mathcal{K}(\mathbb{R}^n)$ with $0\in K$ and any $t\in[0, 1]$, one has
\begin{align}
\label{eq:join}
\left|-t\Delta_pK\vee (1-t)K^p \right|\leq(1-t)^{n(p-1)}|K|^p,
\end{align}
If $\dim(K)=n$ and equality holds for some $t\in(0, 1)$, then $K$ is a simplex with a vertex at the origin.
\end{thm}

A straightforward consequence of Theorem \ref{thm:join} is the following inequality which gives a somewhat weaker upper bound on the mixed volume of  $-\Delta_pK$ and $K^p$ than conjectured by \eqref{KOTR}. For $p=1$, the inequality reduces to \cite[Theorem 1.4]{ArtsteinETAL15}.

\begin{cor}\label{CORE}
Let  $p\in\NN$ and $0< j< n$.   For any  $K\in\mathcal{K}(\mathbb{R}^n)$ one has
$$    V\big(-\Delta_pK[j],K^p[np-j]\big)\leq \frac{n^n}{(n-j)^{n-j}j^j}|K|^p.$$
\end{cor}

Finally, we propose the following unbalanced F\'ary--R\'edei type conjecture, which generalizes \cite[Conjecture 1.2]{ArtsteinETAL15}:
\begin{con}\label{QQ}
    Let $p\in\mathbb{N}.$ For any $K\in\mathcal{K}(\mathbb{R}^n)$ with $\dim(K)=n$, and every $t\in[0, 1]$, there exists $x\in K$ such that
    $$
    \frac{\left|-t\Delta_p(K-x)\vee (1-t)(K-x)^p \right|}{|K|^p}\leq\frac{\left|-t\Delta_pS_n\vee (1-t)(S_n)^p \right|}{|S_n|^p},
    $$
    where $S_n\in\calK(\RR^n)$ is a centered $n$-simplex.
    \end{con}
In Proposition \ref{JKLA} we show that Conjecture \ref{QQ}, if true, implies the inequality \eqref{KOTR} in Conjecture \ref{KK1}. This generalizes \cite[Theorem 1.3]{ArtsteinETAL15}.

We point out that, for the sake of brevity, we again list here only the results obtained in the context of Conjecture \ref{KK1}, that is, related to the bodies $-\Delta_pK$ and $K^p$. Analogous results corresponding to Conjecture \ref{KK2} and to the bodies $-\Delta_pK$ and $\iota_1K,\dots,\iota_pK$ are proven in Section \ref{s:join2}. More precisely, the analogs of Theorem \ref{thm:join}, Corollary \ref{CORE}, Conjecture \ref{QQ}, and Proposition \ref{JKLA} are Theorem \ref{TTTT}, Corollary \ref{COREL}, Conjecture \ref{QQQ}, and Proposition \ref{posledni} respectively.

To conclude the introduction, let us emphasize that although the high-level strategy of the proofs of our main results introduced in \S\ref{ss:intro:weighted} and \S\ref{ss:intro:join} is in general motivated by the proofs of the corresponding results in the case $p=1$ given in \cite{ArtsteinPutterman25} and \cite{ArtsteinETAL15}, respectively, the generalization to $p\geq1$ often requires additional ideas and arguments beyond the basic case. This in particular implies to the inequalities involving the convex bodies $-\Delta_pK$ and $\iota_1K,\dots,\iota_pK$;  notable examples are the proofs of Theorems \ref{T5a} and \ref{TTTT}.

\subsection*{Acknowledgements}
This work is based on Master thesis of the first-named author, defended at Charles University, Faculty of Mathematics and Physics. We are grateful to Eli Putterman for a careful reading of the thesis, for numerous valuable comments and suggestions, and for pointing out reference \cite{Schneider00}.

\section{Preliminaries}

In this section, we summarize some basic facts from convex geometry. For more details, we refer the reader to \cite{Schneider2014}. Throughout the paper, we will consider Euclidean space $\mathbb{R}^n$ equipped with the canonical Lebesgue measure  $|\cdot|$ and assume $n\geq2$. We will denote $\mathbb{R}_+\coloneq[0, \infty)$ and $\mathbb{N}_0\coloneq\{0\}\cup\mathbb{N}$. We will denote by $B^n\subset\R^n$ the closed Euclidean unit ball centered at the origin.

\subsection{Convex bodies}

    A convex body is a nonempty, compact, and convex subset $K\subset\mathbb{R}^n$. The family of all convex bodies in $\mathbb{R}^n$, denoted by $\mathcal{K}(\mathbb{R}^n)$, is closed under Minkowski addition
    $$
    K+L:=\{k+l\mid(k,l)\in K\times L\}
    $$
    as well as under scalar multiplication by $r\in\R$
    $$
    rK:=\{rk\mid k\in K\}.
    $$
    When equipped with the Hausdorff metric
   $$
    d_H(K,L):=\inf\{r>0 \mid K\subset L+ rB^n, L\subset K+ rB^n\},
   $$
   $\mathcal{K}(\mathbb{R}^n)$ becomes a locally compact, complete metric space.

    The reflection of $K\in\mathcal{K}(\mathbb{R}^n)$ about the origin is denoted by $-K:=(-1)K$. Similarly, for $K,L\in\calK(\RR^n)$ and $x\in\RR^n$, we abbreviate $K-L:=K+(-L)$ and $K\pm x:=K+\{\pm x\}$. A convex body $K\in\calK(\RR^n)$ is said to be  symmetric   if $-K=K+x$ for some $x\in\RR^n$; $K\in\calK(\RR^n)$ is said to be of constant width if $K-K=rB^n$ for some $r\in\RR_+$.

   We define the dimension of a convex body $K\in\calK(\R^n)$, written $\dim(K)$, as the dimension of the affine span of $K$. We say that $K\in\mathcal{K}(\mathbb{R}^n)$ is full-dimensional if $\dim(K)=n$ or, equivalently,  $\inter(K)\neq\emptyset$.
    
    Let $K_1, K_2\in\mathcal{K}(\mathbb{R}^n)$. We define the join $K_1\vee K_2\coloneqq \textup{conv}(K_1\cup K_2)$ to be the smallest convex body containing both $K_1$ and $K_2$. This definition extends naturally to finitely many convex bodies.

We will frequently make use of the following result due to Rogers and Shephard.
\begin{lemma}[\protect{\cite[Theorem 1]{RogersShephard58}}]\label{L2}
    Let $1\leq j\leq n-1$ and let $E\subset\RR^n$ be a $j$-dimensional affine subspace. Then for any full-dimensional $K\in\calK(\RR^n)$ one has$$
    |P_{E^{\perp}}K|\cdot|K\cap E|\leq {n \choose j}|K|,$$
where $P_{E^\perp}$ denotes the projection to the orthogonal complement of $E$.  

Moreover, equality holds if and only if for every $a\in E^{\perp}$, the set $K\cap (E+\mathbb{R}_+a)$ is obtained by taking the convex hull of $K\cap E$ and one more point.     In particular, if equality holds, then all nonempty sections $K\cap (E+a)$, $a\in E^{\perp}$, are homothetic.
\end{lemma}

Let us also recall the classical Brunn-Minkowski inequality.

\begin{thm}[{\cite[Theorem 7.1.1]{Schneider2014}}]\label{BM}
    For all $K_1, K_2\in\mathcal{K}(\mathbb{R}^n)$ and $t\in[0,1]$ one has $$\big|(1-t)K_1+tK_2\big|^{\frac{1}{n}}\geq (1-t)|K_1|^{\frac{1}{n}}+t|K_2|^{\frac{1}{n}}.$$
Moreover, if $\dim K_1=\dim K_2=n$, then equality holds for some $t\in (0, 1)$ if and only if $K_1$ and $K_2$ are homothetic.
\end{thm}

Finally, for $K_1, K_2\in\mathcal{K}(\mathbb{R}^n)$ with $\dim K_1=\dim K_2=n$ and $0\in\inter(K_1)$, and $x\in\RR^n$, we define the $K_1$-anisotropic distance as $$d_{K_1}(x, K_2)\coloneqq \min\{r\geq0 : x\in K_2+rK_1\}.$$
Observe that the function $x\mapsto d_{K_1}(x, K_2)$ is continuous. Furthermore, for $r\geq0$, one has $K_2+rK_1=\{x\in\mathbb{R}^n : d_{K_1}(x, K_2)\leq r\}$.

    \subsection{Higher-order difference body}
    \label{ss:higherDK}
    Let $p\in\NN$. Recall that the higher-order difference body of $K\in\calK(\RR^n)$ is given by $D_pK:=-\Delta_pK+K^p$, where $\Delta_p:\RR^n\to(\RR^n)^p$ denotes the diagonal embedding, i.e., $\Delta_p(x)=(x,\dots,x)$ for $x\in\RR^n$. It is straightforward to verify that
$$
D_pK=\left\{(x_1,\dots,x_p)\in(\RR^n)^p\mid K\cap(K-x_1)\cap\cdots\cap(K-x_p)\neq\emptyset\right\},
$$
    see \cite[Proposition 3.1]{K25}. In fact, this was the original definition of Schneider \cite{Schneider70} who, however, used the opposite sign convention, i.e.,$$
\left\{(x_1,\dots,x_p)\in(\RR^n)^p\mid K\cap(K+x_1)\cap\cdots\cap(K+x_p)\neq\emptyset\right\}=\Delta_pK-K^p.
$$
    
Observe that if $\dim K=n$, then $0\in\inter (D_pK)$. Indeed, if $\dim K=n$, there exist $\epsilon>0$ and $x\in\inter (K)$ such that $x+\epsilon B^n\subset K$. Then for any $w\in \epsilon B^n$ one has $x\in K-w$ which shows that $(\epsilon B^n)^p\subset D_pK$.

 \subsection{Simplices}
 
 \label{ss:simplices}
    
   Let $0\leq k\leq n$.  A convex body $K\in\mathcal{K}(\mathbb{R}^n)$ is called a $k$-simplex (or just simplex) if $K=\conv\{v_0, \dots, v_k\}$ is the convex hull of affinely independent points $v_0,\dots,v_k\in\mathbb{R}^n$. We say that a $k$-simplex $\conv\{v_0, \dots, v_k\}$ is centered if $v_0+\dots+v_k=0$.
   
 If $S_n=\conv\{v_0, \dots, v_n\}\subset\RR^n$ is an $n$-simplex, there exist affine functions $A_j\colon \mathbb{R}^n\to\mathbb{R}$, $j=0, \dots, n$, satisfying $\sum_{j=0}^nA_j=1$ and $A_j(v_i)=\delta_{ij}$ such that
 $$S_n=\{x\in\mathbb{R}^n \mid A_j(x)\geq0, j=0, \dots, n\}.$$
   
We will need the following characterization of simplices, first proven for $\lambda=1$ by Rogers and Shephard \cite{RogersShephard57} and later extended to $\lambda\in(0,1)$ by Artstein-Avidan and Putterman \cite{ArtsteinPutterman25}:
\begin{lemma}[{\cite[Lemma 4]{RogersShephard57}} and {\cite[Proposition 10]{ArtsteinPutterman25}}]
\label{lem:simplex}
Let $K\in\mathcal{K}(\mathbb{R}^n)$ with $\dim(K)=n$, and let $\lambda\in (0, 1]$. If $(\lambda K+x)\cap K $ is homothetic to $K$ for all $x\in\inter(K-\lambda K)$, then $K$ is a simplex.  
\end{lemma}

More precisely, we will use the following two easy consequences of Lemma \ref{lem:simplex}:
\begin{corollary}
\label{cor:simplex}
Let $p\in\NN$ and $0<s\leq t\leq1$. If $K\in\mathcal{K}(\mathbb{R}^n)$ with $\dim(K)=n$ is homothetic
\begin{enuma}
\item either to $   K\cap\bigcap_{i=1}^{p}(K-w_i) $ for all $(w_1,\dots,w_p)\in D_pK$, 
\item or to $(tK-w)\cap sK$ for all $w\in\inter(tK-sK)$,
\end{enuma}
then K is a simplex.
\end{corollary}

\begin{proof}
Assume first that the condition (a) holds. If $x\in K-K$, then $K\cap (K+x)\neq\emptyset$. Hence $-\Delta_p(x)\in D_pK$ and so $K$ is homothetic to $K\cap(K+x)$. By Lemma \ref{lem:simplex}, $K$ is a simplex.

Second, assume that (b) holds and take $x\in\inter\big(tK-(st^{-1})tK\big)$. Then $(tK-x)\cap sK$ is homothetic to $K$. Equivalently, $(tK)\cap(sK+x)=(tK)\cap\big((st^{-1})tK+x\big)$ is homothetic to $tK$. Since $\dim( tK)=n$ and $st^{-1}\in(0,1]$, Lemma \ref{lem:simplex} implies that  $K$ is a simplex.
\end{proof}

It is well known (and easy to verify) that if $K=S_n$ is an $n$-simplex, all these properties hold. We will need an explicit description of the intersection from Corollary \ref{cor:simplex}(a) in this case. Choose affine functions $A_j\colon \mathbb{R}^n\to\mathbb{R}$, $j=0,\dots,n$, as above and write $A_j(x)=\alpha_j(x)+\beta_j$, where $\alpha_j\colon\mathbb{R}^n\to\mathbb{R}$ is linear and $\beta_j\in\mathbb{R}$. Then for $\bw=(w_1,\dots,w_p)\in\RR^{np}$ we have
$$
S_n-w_i=\{x\in\mathbb{R}^n \mid A_j(x)\geq -\alpha_j(w_i), j=0, \dots, n\}, \quad i=1, \dots, p.
$$
Thus
$$S_n\cap\bigcap_{i=1}^{p}(S_n-w_i)=\{x\in\mathbb{R}^n \mid A_j(x)\geq m_j(\bw), j=0, \dots, n\},$$
where $m_j(\bw):=\max\{0, -\alpha_j(w_1), \dots, -\alpha_j(w_{p})\}$, $j=0, \dots, n$. Put
\begin{align}
\label{eq:rhoy}
\rho(\bw):=1-\sum_{j=0}^n m_j(\bw)\qquad\text{and}\qquad y(\bw):=\sum_{j=0}^n m_j(\bw)v_j.
\end{align}
Then a straightforward calculation shows that
\begin{align}
\label{eq:Sncap}
S_n\cap\bigcap_{i=1}^{p}(S_n-w_i)=\begin{cases}y(\bw)+\rho(\bw)S_n,&\text{if }\rho(\bw)\geq0;\\\emptyset,&\text{if }\rho(\bw)<0.\end{cases}
\end{align}
Observe also that  for $s>0$ one has $\rho(s\bw)=1-s\big(1-\rho(\bw)\big)$ and $y(s\bw)=sy(\bw)$.

\subsection{Minkowski asymmetry measure}
\label{ss:asymmetry}
Let $K\in\calK(\RR^n)$ be a convex body with $\dim K=n$. The Minkowski asymmetry measure of $K$ is defined by
$$
s(K)\coloneqq \inf\Big\{t>0\mid -K\subset x+tK \text{ for some } x\in\mathbb{R}^n\Big\},
$$
see \cite[\S6.1]{Grunbaum63}. It holds $s(K)\in[1, n]$ and there always exists $x\in\mathbb{R}^n$ such that
$$-K\subset x+s(K)K.$$
Observe that $s(-K)=s(K)$. Moreover, we have $s(K)=1$ if and only if $K$ is symmetric and $s(K)=n$ if and only if $K$ is a simplex. Finally, the map
$$s\colon \{K\in\mathcal{K}(\mathbb{R}^n)\mid \dim(K)=n\}\to[1, n]$$
is continuous in the Hausdorff metric. Observe that these properties together with the monotonicity of mixed volume imply at once the cases $j=1,n-1$ of Conjecture \ref{Godbc}.

\subsection{Mixed volumes}

The mixed volume of convex bodies $K_1, \dots, K_n\in\calK(\RR^n)$ is defined as
$$V(K_1, \dots, K_n)\coloneqq \frac{1}{n!}\left.\frac{\partial^n}{\partial t_1\cdots\partial t_n}\right|_{t_1=\cdots=t_n=0}|t_1K_1+\dots+t_nK_n|.$$ 
According to a classical result of Minkowski, the function $(t_1,\dots,t_n)\mapsto|t_1K_1+\dots+t_nK_n|$ is, in fact, an $n$-homogeneous polynomial on $\R_+^n$.

Mixed volumes satisfy the following identities that will be used throughout the paper. Let $L, K_1, \dots, K_n\in\mathcal{K}(\mathbb{R}^n)$ be convex bodies, $a\geq0$, $\rho\colon\{1, \dots, n\}\to\{1, \dots, n\}$ a permutation, and $T\colon\mathbb{R}^n\to\mathbb{R}^n$ a linear endomorphism. Then
\begin{align*}
V(K_1+\alpha L, \dots, K_n)&= V(K_{1}, \dots, K_{n})+\alpha V(L, K_{2}, \dots, K_{n}),\\
V(L, L, \dots, L)&= |L|,\\
V(K_{\rho(1)}, \dots, K_{\rho(n)})&=V(K_1, \dots, K_n),\\
V\big(T(K_1), \dots, T(K_n)\big)&= |\det(T)| V(K_{1}, \dots, K_{n})
\end{align*}
If $L\subset K_1$, then
$$
V(L, K_2, \dots, K_n)\leq V(K_1, K_2, \dots, K_n);
$$
in particular,
$$
V(K_1,\dots,K_n)\geq0.
$$
Moreover, $V(K_1, \dots, K_n)>0$ if and only if there exist segments $I_i\subset K_i$, $i=1,\dots, n$, with linearly independent directions.

Let us emphasize that throughout the article the Lebesgue measure and the mixed volume will be considered on Euclidean spaces of various dimensions. We believe it should always be clear from the context what the dimension is.

\section{Higher-order Godbersen conjectures in low dimensions}

\label{s:lowdim}

We will first verify Conjectures \ref{KK1} and \ref{KK2} for certain extremal values of parameters $j$ and $j_0,\dots,j_p$, respectively. As we will see, such special cases in particular imply that the conjectures are true in low dimensions.

Let $p\in\NN$. Recall that $\iota_i:\RR^n\to\RR^{pn}$,  $1\leq i\leq p$, denotes the inclusion into the $i$-th factor of $\RR^{pn}=\RR^n\times\cdots\RR^n$, i.e., $\iota_i(x)=(0,\dots,0,x,0,\dots,0)$ for $x\in\RR^n$. Moreover, the following notation will be used for the rest of the paper: We set
$$
\calJ_p^n:=\left\{(j_0,\dots,j_p)\in\NN^{p+1}_0\mid j_0+\cdots+j_p=n\right\},
$$
and for a fixed $K\in\calK(\RR^n)$ we define the function $W_{p,K}:\calJ_p^n\to\RR_+$ by
$$
W_{p,K}(\bj):=V\big(-\Delta_pK[n-j_0],\iota_1K[n-j_1],\dots,\iota_pK[n-j_p]\big),\quad \bj=(j_0,\dots,j_p)\in\calJ_p^n.
$$
For $\bj=(j_0,\dots,j_p)\in\calJ_p^n$ we will also abbreviate
$$
{n\choose \bj}:={n\choose j_0,\dots,j_p}\qquad\text{and}\qquad {np\choose n-\bj}:={np\choose n-j_0,\dots,n-j_p}.
$$

We will repeatedly use the following simple observation:
\begin{lemma}[{\cite[Remarks 3.4 and 3.5]{K25}}]
\label{lem:reduce0}
For each $p\geq2$ and $K\in\calK(\RR^n)$, the function $W_{p,K}$ is symmetric and satisfies
\begin{align*}
W_{p,\,K}(j_0,\dots,j_{p-1},0)={np \choose n}^{-1}|K|\,W_{p-1, \, K}(j_0, \dots, j_{p-1}),\quad (j_0,\dots,j_{p-1})\in\calJ_{p-1}^n.
\end{align*}
\end{lemma}

When studying Conjecture \ref{KK2}, we may thus assume $j_i>0$ for $i=0,\dots,p$ and, in particular, $p\leq n-1$. The case $p=n-1$ of the inequality \eqref{KOTRB} is due to Janson which was mentioned without proof by Schneider \cite[p. 537]{Schneider00}; we include an argument here for completeness.

\begin{prop}[Janson's inequality]\label{janson}
For any $K\in\mathcal{K}(\mathbb{R}^n)$ we have 
\begin{equation}
\label{eq:janson}
    W_{n-1, K}(1,\dots,1)\leq\frac{(n!)^n}{(n^2-n)!}|K|^{n-1}.
\end{equation}
\end{prop}
\begin{proof}
For any $q\in\NN$, Schneider's inequality \eqref{eq:Schneider} yields
$$
        \sum_{\bj\in\mathcal{J}_q^n}
{nq\choose n-\bj} W_{q, K}(\bj)\leq {nq+n \choose n }|K|^q.
$$
Assume $q\geq n-1$ and consider only the summands corresponding to $\bj\in\calJ_q^n\cap\{0,1\}^{q+1}$, i.e., to the multiindices with entries 0 and 1. By Lemma \ref{lem:reduce0}, for each such term we have
$$
{nq\choose n-\bj}W_{q,K}(1,\dots,1,0,\dots,0)={n(n-1)\choose n-1, \dots, n-1}|K|^{q+1-n}W_{n-1, K}(1,\dots,1).
$$
Using this and restricting the above sum to ${q+1\choose n}$ such summands, we get
$$
 {q+1\choose n}{n(n-1)\choose n-1, \dots, n-1}|K|^{q+1-n}W_{n-1, K}(1,\dots,1)\leq {nq+n \choose n}|K|^q.
$$
Letting $q\to\infty$ easily yields \eqref{eq:janson}.
\end{proof}

Note that the limiting argument in the proof of Proposition \ref{janson} does not allow for analysis of equality cases of \eqref{eq:janson}. In fact, we do not know if simplices are the only non-trivial extremizers in this case.

\begin{proof}[Proof of Theorem \ref{thm:main1}(b)]
If $n=2$ then we may assume $p=1$ and $j_0=j_1=1$ which is the known case $j=1$ of Conjecture \ref{Godbc}.

Let $n=3$. The only two cases to consider correspond to the multiindices $(2,1)$ and $(1,1,1)$ for $p=1$ and $p=2$, respectively. The former is again a known case of Conjecture \ref{Godbc}, the latter is a special case of Proposition \ref{janson}.
\end{proof}

By \cite[Proposition 3.3]{K25}, Theorem \ref{thm:main1}(b) implies that the same is true for Conjecture \ref{KK1}. In what follows, we will establish additional cases of Conjecture \ref{KK1} that in particular include the characterization of equality cases in dimension $n=3$, thus proving Theorem \ref{thm:main1}(a). Namely, we will prove the conjecture for $j\in\{0,1,n-1,n\}$. Observe that, unlike in the case $p=1$, for $p\geq 2$ there is in general no symmetry between the mixed volumes $V(K^p[np-j],-\Delta_pK[j])$ and $V(K^p[np-n+j],-\Delta_pK[n-j])$.

First, Conjecture \ref{KK1} is trivial for $j=0$. Second, assuming $p\geq2$, an argument similar to \cite[Remark 3.5(c)]{K25} yields at once 
$$
V(K^p[np-n], -\Delta_pK[n])={np \choose n}^{-1}|K|\,|D_{p-1}K|,
$$ 
which together with Schneider's inequality \eqref{eq:Schneider} proves the case $j=n$. Third, writing $K^p=\sum_{i=1}^p\iota_iK$, expanding by multilinearity, and using Lemma \ref{lem:reduce0}, gives
$$
V(K^p[np-1], -\Delta_p K)= |K|^{p-1}V(K[n-1], -K).$$
Hence, if $j=1$, Conjecture \ref{KK1} is for general $p\in\NN$ equivalent to the known case $p=1$. Finally, since no closed formula for the mixed volume $V(K^p[np-n+1], -\Delta_pK[n-1])$ is known to us, we will deduce the case $j=n-1$ from the following  general statement which might be of independent interest:

\begin{prop}\label{P7}
Let  $p\in\NN$ and $0\leq j\leq n$. For any  $K,L\in\mathcal{K}(\mathbb{R}^n)$ one has 
$$
        {np \choose j}V(K^p[np-j], -\Delta_p L[j])
= {n \choose j}\int\limits_{D_{p-1}K}V(K_{\bw}[n-j], - L[j])\d \bw, 
$$
 where
  $$K_{\bw}:= K\cap\bigcap_{i=1}^{p-1}(K-w_i), \quad \bw=(w_1, \dots, w_{p-1})\in (\mathbb{R}^n)^{p-1}.$$
\end{prop}

\begin{proof}
   Consider the linear map $F\colon\mathbb{R}^{np}\to\mathbb{R}^{np}$ given by
   $$F(x_1, \dots, x_p) \coloneqq (x_1-x_p, \dots, x_{p-1}-x_p, x_p).$$
   It satisfies $|\det(F)|=1$ and $$F(K^p-t\Delta_pL)=\big\{(\bw, z)\in\mathbb{R}^{np-n}\times\mathbb{R}^n \mid z\in K_{\bw}-tL\big\}$$
   for any $t>0$. Consequently, Fubini's theorem yields
   \begin{align*}
            \sum_{j=0}^n{np \choose j}t^jV(K^p[np-j], -\Delta_pL[j])&=|K^p-t\Delta_pL|\\&=\left|F(K^p-t\Delta_pL)\right|\\&=\int\limits_{\mathbb{R}^{np-n}}|K_{\bw}-tL|\d \bw\\&=\sum_{j=0}^n {n \choose  j}t^j\int\limits_{D_{p-1}K}V(K_{\bw}[n-j], -L[j])\d\bw,
        \end{align*}
         where we used the fact that $K_{\bw}\neq\emptyset$ if and only if $\bw\in D_{p-1}K.$ Comparing the coefficients of $t^j$, the claim follows.
\end{proof}

The integral formula of Proposition \ref{P7} gives for general $j$ a (weaker) upper bound than \eqref{KOTR} which, however, does not depend on $p$:

\begin{prop}\label{P9}
Let $p\in\NN$ and $0\leq j\leq n$. For any $K\in\mathcal{K}(\mathbb{R}^n)$ one has
$$V(K^p[np-j], -\Delta_pK[j])\leq n^{\min\{j,\,n-j\}}|K|^p.$$
\end{prop}
\begin{proof}
It suffices to prove the statement for $p\geq2$. We will keep the notation from Proposition \ref{P7}. By Proposition \ref{P7}, we have
\begin{align*}
{np \choose j}V(K^p[np-j], -\Delta_pK[j])&={n \choose j}\int_{D_{p-1}K}V(K_{\bw}[n-j], -K[j])\d \bw \\
&\leq {n \choose j}\int_{D_{p-1}K}V\big(-s(K_{\bw})K_{\bw}[n-j], -K[j]\big)\d \bw \\
&={n \choose j}\int_{D_{p-1}K}s(K_{\bw})^{n-j}V(K_{\bw}[n-j], K[j])\d \bw \\
&\leq n^{n-j}{n \choose j}\int_{D_{p-1}K}V(K_{\bw}[n-j], K[j])\d \bw\\
&= n^{n-j}{np \choose j}V(K^p[np-j], \Delta_pK[j]) \\
&\leq n^{n-j}{np \choose j}V(K^p[np-j], K^p[j]) \\
&= n^{n-j}{np \choose j}|K|^p,
\end{align*}
where for each $\bw\in D_{p-1}K$ we used the existence of $x_{\bw}\in\mathbb{R}^n$ with $K_{\bw}\subset x_{\bw} - s(K_{\bw})K_{\bw}$,
the bound $s(K_{\bw})\leq n$, and the inclusion $\Delta_pK\subset K^p$.

On the other hand, choosing $x\in\mathbb R^n$ such that $-K\subset x+s(K)K$, we have 
$$-\Delta_pK\subset\Delta_p(x)+s(K)\Delta_pK\subset\Delta_p(x)+s(K)K^p.$$
Consequently,
\begin{align*}
V(K^p[np-j],-\Delta_pK[j])\leq s(K)^j V(K^p[np-j],K^p[j])\leq n^j|K|^p.
\end{align*}
\end{proof}

\begin{cor}\label{cor7}
Let $p\in\NN$. For any $K\in\mathcal{K}(\mathbb{R}^n)$ one has
$$V(K^p[np-n+1], -\Delta_pK[n-1])\leq n|K|^p,$$
with equality if and only if $K$ is a simplex, unless $\dim(K)<n$.

\begin{proof}
We may again assume $p\geq2$. By Proposition \ref{P9}, it suffices to show that if $\dim(K)=n$ and $K$ is not a simplex, then the second estimate in the proof of Proposition \ref{P9} (obtained from the bound $s(K_{\bw})\leq n$) is strict. We  again keep the notation from Proposition \ref{P7}.

Since $\inter(K)\neq\emptyset$, we have $0\in \inter(D_{p-1}K)$. Without loss of generality, we may also assume that $0\in \inter(K)$. Then $K_0=K$ is not a simplex, hence $s(K_0)<n$. 

Observe that the mapping $\bw\mapsto K_{\bw}$ is continuous at $0$. Indeed, for any $\varepsilon\in(0, 1)$ it holds that  $(1-\varepsilon)K\subset\inter(K)$. Then for every $\bw$ sufficiently close to $\bw_0$, we have $(1-\varepsilon)K\subset K_{\bw}\subset K$.

Since $K_0=K$ is full-dimensional, it follows that $K_{\bw}$ is full-dimensional for all $\bw$ in some open neighborhood of $0$. Since the map $s$ is continuous on the subset of full-dimensional convex bodies in $\mathcal{K}(\mathbb{R}^n)$, we have $s(K_{\bw})<n$ in some open neighborhood of $0$.
In particular, $s(K_{\bw})<n$ and $V(K_{\bw}, K[n-1])>0$ on a subset of $D_{p-1}K$ of positive $(np-n)$-dimensional measure, which makes the
second inequality in the proof of Proposition \ref{P9} strict and completes the proof.
\end{proof}
\end{cor}

\begin{proof}[Proof of Theorem \ref{thm:main1}(a)]
If $n\in\{2,3\}$, then necessarily $j\in\{0,1,n-1,n\}$. We have seen that Conjecture \ref{KK1} is true in these cases.
\end{proof}

\section{Higher-order Godbersen conjectures for special classes of convex bodies}

In this section we will use an argument due to Godbersen \cite{Godbersen} to show that Conjectures \ref{KK1} and \ref{KK2} hold for convex bodies of constant width if $n\leq6$, in the latter case for $n=6$ possibly without the characterization of equality cases. Along the way, we will also observe that in general dimension the conjectures hold for symmetric convex bodies; note that although this is rather easy to see, it is not as obvious as in the classical case $p=1$.

Recall from \cite{Godbersen} that if $K\in\calK(\RR^n)$ is of constant width, then there exists
$x\in\mathbb{R}^n$ such that
$$-K \subset x + \alpha_n K,$$
where
$$\alpha_n := \frac{n+\sqrt{2n(n+1)}}{n+2}.$$
Observe that in particular $s(K)\leq\alpha_n$. Moreover, this constant is optimal in the sense that there exists $\widetilde K\in\calK(\RR^n)$ of constant width for which $s(\widetilde K)=\alpha_n.$ Godbersen's proof of Conjecture \ref{Godbc} for bodies of constant width is based on following estimate:

\begin{lemma}[{\cite[\S7]{Godbersen}}]
\label{lem:godbersen}
For any $1\leq j\leq\lfloor\frac n2\rfloor$ it holds that
$$(\alpha_n)^{j} < {n \choose j}.$$
\end{lemma}

\begin{proof}[Proof of Theorem \ref{thm:main1}(c)]
Observe that for any $K\in\calK(\RR^n)$ we have
$$V\big(-\Delta_pK[j],K^p[np-j]\big)\leq s(K)^j|K|^p.$$
Hence,  the statement for $j\leq\lfloor\frac n2\rfloor$ follows at once from Lemma \ref{lem:godbersen} and the fact that $s(K)\leq\alpha_n$ for bodies of constant width. Proceeding to the case $n\leq 6$, by the results of Section \ref{s:lowdim} we may assume $2\leq j\leq n-2$ and $n\in\{4,5,6\}$. A short calculation shows that the desired inequalities
$$
s(K)^j< {n \choose j}, \qquad j=2, \dots, n-2,
$$
hold if and only if
$$
s(K)<\beta_n\coloneqq {n \choose n-2}^{\frac{1}{n-2}}.
$$
Using $s(K)\leq\alpha_n$, it suffices to verify that $\alpha_n<\beta_n$ for $n\in\{4, 5, 6\}$. To this end, we compute
\begin{align*}
  \alpha_4=\frac{2+\sqrt{10}}{3}\approx1.721&<2.449\approx\sqrt{6}=\beta_4 \\ \alpha_5=\frac{5+2\sqrt{15}}{7}\approx1.821&<2.154\approx10^{\frac{1}{3}}=\beta_5 \\ \alpha_6=\frac{3+\sqrt{21}}{4}\approx1.896&<1.968\approx15^{\frac{1}{4}}=\beta_6.
\end{align*}
\end{proof}

Observe that in higher dimensions this method fails as
$$
\alpha_n\geq\alpha_7=\frac{7+4\sqrt{7}}{9}\approx1.953>1.838\approx21^{\frac{1}{5}}=\beta_7\geq \beta_n, \qquad n\geq 7.
$$
We will now apply an analogous argument to the finer coefficients in Conjecture~\ref{KK2}.

\begin{proof}[Proof of Theorem \ref{thm:main1}(d)]
 Let $p \in \mathbb{N}$ and consider a full-dimensional convex body $K \in \mathcal{K}(\mathbb{R}^n)$. Since $-K\subset x+s(K)K$ for some $x\in\RR^n$, $-\Delta_pK$ is similarly contained in a translate of $s(K)K^p$. Consequently, for any $\bj=(j_0,\dots,j_p)\in\mathcal{J}^n_p$ we obtain
   \begin{align*}
           \frac{W_{p,K}(\bj)}{s(K)^{n-j_0}}&\leq V(K^p[n-j_0], \iota_1K[n-j_1], \dots, \iota_pK[n-j_p])\\
           &= \sum_{(k_1,\dots,k_p)\in\mathcal{J}^{n-j_0}_{p-1}}{n-j_0 \choose k_1, \dots, k_p}V(\iota_1K[n-j_1+k_1], \dots, \iota_pK[n-j_p+k_p])\\
        &= {n-j_0 \choose j_1, \dots, j_p}V(\iota_1K[n], \dots, \iota_pK[n])\\&= \frac{(n!)^p}{(np)!}{n-j_0 \choose j_1, \dots, j_p}|K|^p.
       \end{align*}
        Using the symmetry of $W_{p,K}$, we conclude
        \begin{align}
        \label{ineq}
           W_{p, K}(\bj)\leq \frac{(n!)^{p}}{(np)!}M_{s(K)}(\bj){n \choose \bj}|K|^p,
       \end{align}
       where for $\bj=(j_0,\dots,j_p)\in\mathcal{J}^n_p$ and $t\in[1, n]$ we set
   $$M_{t}(\bj):= \min_{0\leq l\leq  p}{n \choose j_l}^{-1}t^{n-j_l}.$$
Observe that, for fixed $\bj$, $M_{t}(\bj)$ is increasing in $t$.

Assume $K$ is of constant width. In this case $s(K)\leq\alpha_n$ and hence Conjecture \ref{KK2} holds for $\bj\in\calJ_p^n$ if $M_{\alpha_n}(\bj)<1$. We only need to consider the cases not covered by Section \ref{s:lowdim}.

If $n=3$, the only remaining case is $p=2$ and $\bj=(1, 1, 1)$. In this situation, we have
$$M_{\alpha_3}(1,1,1)=\frac{1}{3}\alpha_3^2=\frac{33+12\sqrt{6}}{75}\approx0.832<1.$$

To prove \eqref{KOTRB} for $n\in\{4,5,6\}$, it remains to consider $\bj\in\calJ_p^n$ with $1\leq j_0,\dots,j_p\leq n-2$. Then, using the symmetry of $W_{p,K}$, we may assume that $j_0\geq 2$ and hence the computations from the proof of Theorem \ref{thm:main1}(c) imply
$$
M_{\alpha_n}(\bj)\leq {n \choose j_0}^{-1}\alpha_n^{n-j_0}\leq{n \choose 2}^{-1}\alpha_n^{n-2}<1.
$$
\end{proof}

Note that this method again fails for $n\geq 7$, since in this case we have 
$$M_{\alpha_n}(2,1,\dots,1)={n \choose 2}^{-1}\alpha_n^{n-2}>1.$$
We conclude the section by observing that Conjecture \ref{KK2} holds for symmetric convex bodies.
    \begin{proof}[Proof of Theorem \ref{thm:main1}(e)]
    Since a symmetric convex body $K\in\calK(\RR^n)$ with $\dim K=n$ satisfies $s(K)=1$, one has $M_{s(K)}(j_0,\dots,j_p)<1$ and hence the claim follows at once from \eqref{ineq}.
    \end{proof}

\section{Stability of the higher-order Godbersen conjectures under products}

In this section we will show that Conjectures \ref{KK1} and \ref{KK2} are stable under the operation of taking Cartesian products. More precisely, we will prove that if $K_i\in\calK(\RR^{n_i})$, $i=1,2$, satisfy Conjecture \ref{KK2} for $p\in\NN$, then $K_1\times K_2\in\calK(\RR^{n_1+n_2})$ satisfies  Conjecture \ref{KK2} for this value of $p$. The proof of the corresponding statement for Conjecture \ref{KK1} is analogous (and simpler) and so we omit it.

Throughout this section, we will denote $n:=n_1+n_2$ and, for simplicity, we will interchangeably regard the space $\RR^{pn}$ both as $(\RR^{n_1}\times\RR^{n_2})^p$ and $(\RR^{n_1})^p\times (\RR^{n_2})^p$. Moreover, for any $\bj\in\calJ_p^n$ we will denote
$$
\calI(\bm{j}):=\left\{(\bj^1,\bj^2)\in\mathcal{J}^{n_1}_p\times\mathcal{J}^{n_2}_p \mid  \bj^1+\bj^2=\bj\right\}.$$

We will first show that the higher-order difference body and its mixed-volume coefficients behave naturally under products.
\begin{lemma}\label{L5}
    Let $p\in\NN$. For any $K_1\in\mathcal{K}(\mathbb{R}^{n_1})$ and $K_2\in\mathcal{K}(\mathbb{R}^{n_2})$ one has
    \begin{align}
    \label{eq:Dproduct}
    D_p(K_1\times K_2)=D_pK_1\times D_pK_2.
    \end{align}
    Moreover, for any $\bj\in\mathcal{J}^n_p$, one has
     \begin{align}
     \label{eq:Wproduct}
            {np \choose n-\bj}\,W_{p, K_1\times K_2}(\bj)
               =\sum\limits_{(\bm{j}^1,\,\bm{j}^2)\in\mathcal{I}(\bm{j})}{n_1p \choose n_1-\bj^1}{n_2p \choose n_2-\bj^2}\,
               W_{p,K_1}(\bm{j}^1)W_{p,K_2}(\bm{j}^2).
    \end{align} 
\end{lemma}

    \begin{proof}
    Denote $K:=K_1\times K_2$. With the above identification, we have $\Delta_pK=\Delta_pK_1\times\Delta_pK_2$ and $\iota_iK=\iota_iK_1\times\iota_iK_2$, $i=1,\dots,p$. Hence for any $\mu_1, \dots, \mu_p>0$, we get
            \begin{align*}
                -\Delta_pK+\sum_{l=1}^p\mu_l\iota_lK=\left(-\Delta_pK_1+\sum_{l=1}^p\mu_l\iota_lK_1\right)\times\left(-\Delta_pK_2+\sum_{l=1}^p\mu_l\iota_lK_2\right).
            \end{align*}
 Putting $\mu_1=\cdots=\mu_p=1$ proves \eqref{eq:Dproduct}. On the other hand, taking the volume and polarizing easily yields \eqref{eq:Wproduct}.
    \end{proof}

The following observation will later be used to estimate the sum in \eqref{eq:Wproduct}.

\begin{lemma}\label{L6}
    Let $p\in\NN$. For any $\bm{j}\in\mathcal{J}_p^n$ satisfying $j_0, \dots, j_p< n$ we have
    $$
    \sum\limits_{(\bm{j}^1,\,\bm{j}^2)\in\mathcal{I}(\bm{j})}\prod_{l=0}^p{n_1 \choose j_l^1}\prod_{l=0}^p{n_2 \choose j_l^2}<\prod_{l=0}^p{n \choose j_l}.$$
    \begin{proof}
    Using the usual convention ${a\choose b}:=0$ for $b>a$, we have\begin{align*}
            \prod_{l=0}^p{n \choose j_l}&=\prod_{l=0}^p\sum_{j^1_l+j^2_l=j_l}{n_1 \choose j^1_l}{n_2 \choose j^2_l}\\
            &=\sum_{\substack{\bm{j}^1,\,\bm{j}^2\in\mathbb{N}_0^{p+1}\\
\bm{j}^1+\bm{j}^2=\bm{j}}}
\prod_{l=0}^{p}{n_1\choose j_l^1}
\prod_{k=0}^{p}{n_2\choose j_k^2}\\
&\geq\sum\limits_{(\bm{j}^1,\,\bm{j}^2)\in\mathcal{I}(\bm{j})}\prod_{l=0}^p{n_1 \choose j_l^1}\prod_{k=0}^p{n_2 \choose j_k^2}.
        \end{align*}
We will prove that the last inequality is in fact strict by showing that the penultimate sum contains at least one positive summand that is omitted in the last sum. For each $l\in\{0, \dots, p\}$, the values of $j_l^1$ for which the corresponding summand is positive are exactly the integers from $$\mathbb{N}_0\cap \left[\max\{0, j_l-n_2\}, \min\{n_1, j_l\}\right].$$ Since $j_0, \dots, j_p< n$, there exists an index $j_{l_0}$ with $1\leq j_{l_0}\leq n-1$ and hence the number of such integers is at least two in this case. Consequently, there exist two distinct pairs of non-trivial summands corresponding to $(\bj^1,\bj^2)$ and $(\tilde\bj^1,\tilde\bj^2)$, respectively, where $\bj^1$ differ from $\tilde\bj^1$ only in the $l_0$-th coordinate (and so does $\bj^2$ from $\tilde\bj^2$). This means that at least one of the tuples $\bj^1$ and $\tilde\bj^1$ is not an element of $\mathcal{J}_p^{n_1}$.
    \end{proof}
\end{lemma}

We can now proceed to the proof of the stability of Conjecture \ref{KK2} under Cartesian products.

    \begin{proof}[Proof of Theorem \ref{thm:main1}(f)]
   Let $p\in\NN$ and assume that $K_1\in\calK(\RR^{n_1})$ and $K_2\in\calK(\RR^{n_2})$ satisfy Conjecture \ref{KK2} for such $p$, that is, for all $\bj^1\in\calJ_p^{n_1}$ and $\bj^2\in\calJ_p^{n_2}$, respectively.  Let $\bm{j}\in\mathcal{J}^n_p$. We may assume  $j_0, \dots, j_p<n$ and also that $\dim(K_1)=n_1$ and $\dim(K_2)=n_2$. Then Lemma \ref{L5}, the assumption on $K_1$ and $K_2$, and Lemma \ref{L6} yield
         \begin{align*}
               {np \choose n-\bj}W_{p,\, K}(\bm{j})&=\sum\limits_{(\bm{j}^1,\,\bm{j}^2)\in\mathcal{I}(\bm{j})}{n_1p \choose n_1-\bj^1}{n_2p \choose n_2-\bj^2}\,
               W_{p,K_1}(\bm{j}^1)W_{p,K_2}(\bm{j}^2)\\
               &\leq \sum\limits_{(\bm{j}^1,\,\bm{j}^2)\in\mathcal{I}(\bm{j})}\frac{(n_1!)^p}{(n_1p)!}{n_1p \choose n_1-\bj^1}{n_1 \choose \bj^1}\frac{(n_2!)^p}{(n_2p)!}{n_2p \choose n_2-\bj^2}{n_2 \choose \bj^2}|K_1|^{p}|K_2|^{p}\\
               &=|K|^{p}\sum\limits_{(\bm{j}^1,\,\bm{j}^2)\in\mathcal{I}(\bm{j})}\prod_{l=0}^p{n_1 \choose j_l^1}\prod_{l=0}^p{n_2 \choose j_l^2}\\
               &<|K|^{p}\prod_{l=0}^p{n \choose j_l}\\
               &=\frac{(n!)^p}{(np)!}{np \choose n-\bj}{n \choose \bj}|K|^{p}.
            \end{align*}
            Since $K_1\times K_2$ is not a simplex, it follows that it satisfies Conjecture \ref{KK2} for $p$. The corresponding statement for Conjecture \ref{KK1} is proven analogously.
    \end{proof}

\section{Higher-order difference body and weighted inequalities for $-\Delta_p K$ and $K^p$}

In this section we deal with two consequences of Conjecture \ref{KK1}, namely, sharp inequalities for two different weighted sums of the mixed volumes of $-\Delta_p K$ and $K^p$, see Theorems \ref{T4} and \ref{T5}. The former is proven here, the latter will be deduced in the next section as a consequence of a more general statement. We also consider here a higher-order generalization of the unbalanced difference body and show that the corresponding weighted Schneider--Rogers--Shephard type inequality (Conjecture \ref{COn}), which is weaker than Conjecture \ref{KK1}, would still imply both Theorems \ref{T4} and \ref{T5}.

The following notation will be used throughout the entire section: For $p\in\mathbb{N}$, $t\in[0, 1]$, and $K\in\mathcal{K}(\mathbb{R}^n)$ we define
\begin{align}
\label{eq:defCptK}
C_{p, t, K}\coloneqq \big(\{0\}\times (1-t)K^p\big) \vee (\{1\}\times -t\Delta_pK)\in\calK(\mathbb{R}\times\mathbb{R}^{np}).
\end{align}
Equivalently, one has
    $$C_{p, t, K}=\left\{\big(s, (1-s)(1-t)y-st\Delta_p(x)\big) \mid s\in[0, 1],\, x\in K, \, y\in K^p\right\}.$$
We will also repeatedly use the following well-known integral formula for $k,l\in\NN_0$:
\begin{align}
\label{eq:intcombinat1}
\int_0^1s^k(1-s)^l\d s=\frac{k!l!}{(k+l+1)!}.
\end{align}

\subsection{Weighted inequalities I}

In order to prove Theorem \ref{T4}, we will need two lemmas.

\begin{lemma}
\label{lem:CptK}
Let $p\in\mathbb{N}$. For any $K\in\mathcal{K}(\mathbb{R}^n)$ and  $t\in[0, 1]$ we have
     \begin{align}
     \label{eq:CptK}
     |C_{p, t, K}|= \frac{1}{np+1}\sum_{j=0}^n(1-t)^{np-j}t^jV(K^p[np-j], -\Delta_pK[j]).
     \end{align}
\end{lemma}

\begin{proof}
   We have
        \begin{align*}
              |C_{p, t, K}|&=\int\limits_{0}^1\left|(1-s)(1-t)K^p-st\Delta_pK\right|\d s\\
              &=\sum_{j=0}^n{np \choose j}(1-t)^{np-j}t^jV(K^p[np-j], -\Delta_pK[j])\int\limits_{0}^1(1-s)^{np-j}s^j\d s\\
              &=\frac{1}{np+1}\sum_{j=0}^n(1-t)^{np-j}t^jV(K^p[np-j], -\Delta_pK[j]).
        \end{align*}
\end{proof}

\begin{lemma}\label{L3}
     Let $p\in\mathbb{N}$. For any $K\in\mathcal{K}(\mathbb{R}^n)$ and  $t\in[0, 1]$ we have
     \begin{align}
         \label{eq:L3}
         |C_{p, t, K}|\leq \frac{(1-t)^{n(p-1)}}{np+1}|K|^p.
         \end{align}
     \begin{proof}
         Let $t\in(0, 1)$. Set 
         \begin{align*}
                 T_{p, t, K}\coloneqq \{(0, 0, y) : y\in (1-t)K^p\}\vee\big\{\big(1, x, -\Delta_p(x)\big) : x\in tK\big\}\in\mathcal{K}(\mathbb{R}\times\mathbb{R}^n\times\mathbb{R}^{np}).
             \end{align*}
                   Since
         \begin{align*}
                 T_{p,t, K}&=\big\{\big(s, sx, -s\Delta_p(x)+(1-s)y\big) : x\in tK,\, y\in (1-t)K^p,\, s\in [0, 1] \big\}\\&=\big\{(s, w, z) : w\in stK,\, z+\Delta_p(w)\in (1-s)(1-t)K^p,\, s\in [0, 1] \big\},
         \end{align*}
we have \begin{align}\label{Eq70}
        |T_{p, t, K}|=|tK|\cdot|(1-t)K^p|\int\limits_{0}^1(1-s)^{np}s^n\d s=\frac{np+1}{{np+n+1 \choose n}}t^n(1-t)^{np}|K|^{p+1}.
    \end{align}
Next, we consider the section of $T_{p, t, K}$ by the affine space $$E_{t}\coloneqq \{(1-t, v, 0) : v\in\mathbb{R}^n\}\subset\mathbb{R}\times\mathbb{R}^n\times\mathbb{R}^{np}$$ and the orthogonal projection of $T_{p, t, K}$ onto $E_{t}^{\perp}$. To this end, we obtain
$$T_{p, t, K}\cap E_t=\{(1-t, v, 0) : v\in t(1-t)K\}$$
and
\begin{align*}
        P_{E_{t}^{\perp}}(T_{p, t, K})&=\{(s, 0, y) :  (s, x, y)\in T_{p,t, K} \text{ for some }x\in\RR^n\}\\
        &=\big\{(s, 0, y) : (st\Delta_pK)\cap\big((1-s)(1-t)K^p-y\big)\neq\emptyset\big\}\\
        &=\big\{(s, 0, y) : y\in(1-s)(1-t)K^p-st\Delta_pK\big\}\\
        &=C_{p, t, K}.
    \end{align*}
Applying Lemma \ref{L2} together with \eqref{Eq70}, it follows
$$
        |C_{p, t, K}|\leq {np+n+1 \choose n}\frac{|T_{p, t, K}|}{t^n(1-t)^n|K|}=\frac{(1-t)^{n(p-1)}}{np+1}|K|^p.
        $$
This proves the claim for $t\in(0, 1)$. For $t\in\{0, 1\}$, the inequality follows by continuity.
     \end{proof}
\end{lemma}

    \begin{proof}[Proof of Theorem \ref{T4}]
   
   First, the inequality \eqref{Eq71} follows at once by combining \eqref{eq:CptK} with \eqref{eq:L3}.
   
   Second, assume that $\dim(K)=n$ and that equality holds in \eqref{Eq71} for some $t\in (0, 1)$. Then we have equality in \eqref{eq:L3} and the proof of Lemma \ref{L3} together with Lemma \ref{L2} in turn imply that all non-empty sections of the body $T_{p,t ,K}$ by the affine subspaces of the form $E_t+(0, 0, \ba)$, where $\ba=(a_1, \dots, a_p)\in(\mathbb{R}^n)^p$, are homothetic. Since \begin{align*}
            T_{p, t, K}\cap\big(E_t+(0, 0, \ba)\big)&=\big\{\big(1-t, t(1-t)x, \ba\big) : x\in K,\, t(1-t)\Delta_p(x)\in t(1-t)K^p-\ba \big\}\\&=\left\{(1-t, x, a) : x\in \big(t(1-t)K\big)\cap\bigcap_{l=1}^p\big(t(1-t)K-a_l\big)\right\}\\&=\{1-t\}\times\Big(\big(t(1-t)K\big)\cap\bigcap_{l=1}^p\big(t(1-t)K-a_l\big)\Big)\times\{\ba\},
    \end{align*} it follows that the bodies 
    $$\big(t(1-t)K\big)\cap\bigcap_{l=1}^p\big(t(1-t)K-a_l\big), \quad \ba\in t(1-t)D_pK,$$ are all mutually homothetic. By Corollary \ref{cor:simplex}, $K$ thus has to be a simplex.
    \end{proof}

\subsection{Higher-order unbalanced difference body I}
Recall from the introduction that for any $p\in\NN$ and $t\in[0,1]$ we define the higher-order unbalanced difference body of $K\in\calK(\RR^n)$ as $D_p^tK= (1-t)K^p-t\Delta_pK$
and observe that
    \begin{align*}
        |D_p^tK|        =(1-t)^{n(p-1)}\sum_{j=0}^n {np \choose j} (1-t)^{n-j}t^j V(K^p[np-j], -\Delta_pK[j])
\end{align*}
In particular, if $K=S_n$ is an $n$-simplex, we have
\begin{equation}
\label{f0}
    |D_p^tS_n|=(1-t)^{n(p-1)}\sum_{j=0}^n {np \choose j}{n \choose j}(1-t)^{n-j}t^j |S_n|^p.
\end{equation}
The inequality \eqref{EQ47ll} of Conjecture \ref{COn} can thus be equivalently stated as
$$
\sum_{j=0}^n{np \choose j}(1-t)^{n-j}t^j V(K^p[np-j], -\Delta_pK[j])\leq \sum_{j=0}^n {np \choose j}{n \choose j}(1-t)^{n-j}t^j |K|^p,\quad t\in[0, 1].
$$       
In this formulation it is obvious that Conjecture \ref{COn} would follow from Conjecture \ref{KK1}. We will now show that even Conjecture \ref{COn} implies Theorems \ref{T4} and \ref{T5}.

\begin{prop}
\label{pro:impliesT4}
    Conjecture \ref{COn} implies Theorem \ref{T4}.
\end{prop}
\begin{proof}
Fix $p\in\mathbb{N}$ and $t\in(0,1)$, and for $s\in[0,1]$ set $ \alpha_t(s)\coloneqq (1-s)(1-t)+st $ and $ \mu_t(s)\coloneqq\frac{st}{\alpha_t(s)}$.
Then $\mu_t(s)\in[0,1]$ and we can write
\begin{align}
\label{eq:volCptk}
            |C_{p, t, K}|=\int\limits_{0}^1|(1-s)(1-t)K^p-st\Delta_pK|\d s=\int\limits_{0}^1\alpha_t(s)^{np}\left|D_p^{\mu_t(s)}K\right|\d s.
    \end{align}
    Using \eqref{eq:volCptk}, \eqref{f0}, and assuming that $K$ satisfies \eqref{EQ47ll}, we have
    \begin{align*}
        |C_{p, t, K}|&\leq \frac{|K|^p}{|S_n|^p}\int\limits_{0}^1\alpha_t(s)^{np}\left|D_p^{\mu_t(s)}S_n\right|\d s\\
        &=|K|^p\sum_{j=0}^n {np \choose j}{n \choose j} \int\limits_{0}^1\alpha_t(s)^{np}\big(1-\mu_t(s)\big)^{np-j}\mu_t(s)^j\d s\\
        &=|K|^p\sum_{j=0}^n {np \choose j}{n \choose j} (1-t)^{np-j}t^j\int\limits_{0}^1(1-s)^{np-j}s^j\d s\\
        &=\frac{(1-t)^{n(p-1)}}{np+1}|K|^p.
\end{align*}
Together with \eqref{eq:CptK}, this immediately implies \eqref{Eq71}. If $t\in\{0, 1\}$, the inequality is trivial.

Let $S_n\in\calK(\RR^n)$ be an $n$-simplex and assume that $|K|=|S_n|>0$. If equality holds in \eqref{Eq71} for some $t\in (0, 1)$, then
$$
\left|D_p^{\mu_t(s)}K\right|=\left|D_p^{\mu_t(s)}S_n\right|
$$
for some $s\in(0,1)$ and Conjecture \ref{COn}, if true, implies that $K$ is an $n$-simplex.
\end{proof}

\begin{prop}
\label{pro:impliesT5}
    Conjecture \ref{COn} implies Theorem \ref{T5}.
\end{prop}
\begin{proof}
Fix $p\in\NN$ and $t\in(0,1)$ and keep the notation from Theorem \ref{T5}. First, observe that
\begin{align*}
        c_{n,  p,  j}(t)=n {np \choose j} \int\limits_{0}^1 (1-s)^{n-1} (1-t+st)^{np-j}(st)^j\d s.
    \end{align*}
    Consequently, we have
    $$
    \sum_{j=0}^n c_{n, \, p, \, j}(t) V(K^p[np-j], -\Delta_pK[j])=n\int\limits_{0}^1(1-s)^{n-1}\beta_{t}(s)^{np}\left|D^{\nu_t(s)}_pK\right|\d s,
    $$
  where for $s\in[0,1]$ we set  $  \beta_t(s)\coloneqq (1-t+st)+st $ and $\nu_t(s)\coloneqq\frac{st}{\beta_t(s)}$. The rest of the proof is completely analogous to the proof of Proposition \ref{pro:impliesT4}, using that $\sum_{j=0}^n {n \choose j} c_{n, p, j}(t)=1$.
\end{proof}

\section{Higher-order difference body and weighted inequalities for $-\Delta_pK$ and $\iota_1K, \dots, \iota_pK$}
\label{s:higher2}

Now we will generalize the results of Artstein-Avidan and Putterman \cite{ArtsteinPutterman25} to the context of Conjecture \ref{KK2}. First, we will prove two consequences of the conjecture, namely, sharp inequalities for two different weighted sums of the mixed volumes of $-\Delta_pK$ and $\iota_1K, \dots, \iota_pK$, and deduce Theorem \ref{T5} from one of them. Second, we will define the corresponding higher-order unbalanced difference body and show that phenomena analogous to those proven in the previous section persist in the context of Conjecture \ref{KK2}. Namely, we will conjecture the corresponding weighted Schneider--Rogers--Shephard type inequality, weaker than Conjecture \ref{KK2}, and we will show that it still implies the weighted inequalities.

Let us establish a notation that will be used throughout the entire section. Fix $p\in\NN$.  First, we will keep the notation $\calJ_p^n$, $W_{p,K}$, ${n\choose \bj}$, and ${np\choose n-\bj}$ from Section \ref{s:lowdim}. Further, we will consider the standard $p$-simplex in $\RR^{p+1}$ given by
$$\Theta_p:=\left\{(t_0, \dots, t_p)\in\mathbb{R}^{p+1}_+\mid t_0+\cdots+t_p=1\right\}$$
and for $\bt=(t_0,\dots,t_p)\in\Theta_p$ we will denote
$$
\hat\bt:=(t_1,\dots,t_p)\qquad\text{and}\qquad\hat\Theta_p:=\{\hat\bt\mid \bt\in\Theta_p\}.
$$
Observe that the map $\Theta_p\to\hat\Theta_p:\bt\mapsto\hat\bt$ is a bijection. When $\bj=(j_0,\dots,j_p)\in\calJ_p^n$, $k\in\NN_0$, and $\bk=(k_0,\dots,k_p)\in(\NN_0)^{p+1}$, we will also abbreviate
$$
\bt^{n-\bj}:=\prod_{l=0}^pt_l^{n-j_l},\qquad \bt^k:=\prod_{l=0}^pt_l^k,\qquad\text{and}\qquad   \bt^{\bk}:=\prod_{l=0}^pt_l^{k_l}.
$$
Let $e_1,\dots, e_p$ be the standard basis of $\RR^p$. For any $K\in\calK(\RR^n)$ and $\bt\in\Theta_p$ we define
\begin{align}
\label{eq:defGptK}
G_{p,\bt, K}:= (\{0\}\times-t_0\Delta_pK)\vee\bigvee_{l=1}^p (\{e_{l}\}\times t_l\iota_lK)\in\calK(\RR^p\times\RR^{pn}).
\end{align}
Equivalently,
$$
G_{p,\bt, K}=\left\{ \big(\hat\bs, -s_0t_0\Delta_p(x_0)+s_1t_1\iota_1(x_1)+\dots+s_pt_p\iota_p(x_p)\big)\mid x_0,\dots,x_p\in K,\bs\in\Theta_p\right\}.
$$
We will also repeatedly use the following formula which is a direct generalization of \eqref{eq:intcombinat1}:
\begin{align}
\int_{\hat\Theta_p}\bs^{\bk} \d\hat\bs=\frac{k_0!\cdots k_p!}{(k_0+\cdots+k_p+p)!},\quad \bk\in(\NN_0)^{p+1}.
\end{align}

\subsection{Weighted inequalities II}

\begin{lemma}
Let $p\in\NN$. For any $K\in\calK(\RR^n)$ and $\bt\in\Theta_p$ we have
\begin{align}
\label{eq:GptK}
 |G_{p, \bt, K}|=\frac{(np)!}{(np+p)!}\sum_{\bm{j}\in\mathcal{J}^n_p}\bt^{n-\bj}W_{p, K}(\bm{j}).
\end{align}
\end{lemma}

\begin{proof}
A direct computation shows
        \begin{align*}
                 |G_{p, \bt, K}|&=\int_{\hat\Theta_p}\left|-s_0t_0\Delta_pK+\sum_{l=1}^ps_lt_l\iota_lK\right|\d\hat \bs\\
                 &=\sum_{\bm{j}\in\mathcal{J}^n_p}{np \choose n-\bj}\bt^{n-\bj}W_{p, K}(\bm{j})\int_{\hat\Theta_p}\bs^{n-\bj}\d \hat\bs\\
                 &=\frac{(np)!}{(np+p)!}\sum_{\bm{j}\in\mathcal{J}^n_p}\bt^{n-\bj}W_{p, K}(\bm{j}).
        \end{align*} 
\end{proof}

\begin{lemma}\label{L3a}
   Let $p\in\NN$. For any $K\in\calK(\RR^n)$ and $\bt\in\Theta_p$ we have
\begin{align}
   \label{eq:L3a}
|G_{p,\bt, K}|\leq\frac{(n!)^{p}}{(np+p)!}\sum_{\bm{j}\in\mathcal{J}^n_p} {n \choose \bj}\bt^{n-\bj}|K|^p.
   \end{align}
     \begin{proof}
     We may assume $\bm{t}\in\relint(\Theta_p)$, the general case follows by continuity. Set
     \begin{align*}
                 H_{p,\bt, K}:=\big\{\big(0, t_0x_0, -t_0\Delta_p(x_0)\big) \mid  x\in K\big\}\vee\bigvee_{l=1}^p \{\big(e_{l}, 0, t_l\iota_l(x_l)\big) : x_l\in K\}\in\calK(\RR^p\times\RR^n\times\RR^{np}).
         \end{align*}
       Equivalently,
       $$
       H_{p, \bt, K}=\left\{\big(\hat\bs, s_0t_0x_0, -s_0t_0\Delta_p(x_0)+s_1t_1\iota_1(x_1)+\cdots+ s_pt_p\iota_p(x_p)\big)\mid x_0,\dots,x_p\in K,\bs\in\Theta_p\right\}
$$
        and so we have
        \begin{align}
        \label{Eq70a}
|H_{p, \bt, K}|=|K|^{p+1}\bt^n\int_{\hat\Theta_p}\bs^n\d\hat\bs=|K|^{p+1}\frac{(n!)^{p+1}}{(np+n+p)!}\bt^n.
    \end{align}
Next, we consider the section of $H_{p, \bt, K}$ by the affine subspace
$$
E_{\bt}:=\{(\hat\theta_{\bm{t}}, v, 0) \mid v\in\mathbb{R}^n\},
$$
where $ \theta_{\bm{t}}:=\left(\sum_{l=0}^pt_l^{-1}\right)^{-1}\left(t_0^{-1}, \dots,t_p^{-1}\right)$, and the orthogonal projection onto $E_{\bt}^{\perp}$. We obtain
    \begin{align*}
        H_{p, \bt, K}\cap E_{\bt}=\left\{\left(\hat\theta_{\bt}, \left(\sum_{l=0}^pt_l^{-1}\right)^{-1}x, 0\right) \mid x\in K\right\}
    \end{align*}
    and
\begin{align*}
        P_{E_{\bt}^\perp}(H_{p,\bt, K})&=\{(\hat\bs, 0, y) : (\hat\bs, x, y)\in H_{p, \bt, K} \text{ for some }x\in\RR^n\}\\
        &=\{(\hat\bs, 0, y) : y\in -s_0t_0\Delta_pK+s_1t_1\iota_1K+\dots+s_pt_p\iota_pK\}\\
        &=G_{p,\bt, K}.
\end{align*}
It follows $|H_{p, \bt, K}\cap E_{\bt}|=\left(\sum_{l=0}^pt_l^{-1}\right)^{-n}|K|$ and $|P_{E_{\bm{t}}^{\perp}}(H_{p, \bm{t}, K})|=|G_{p, \bm{t}, K}|$. Applying Lemma \ref{L2} together with \eqref{Eq70a}, we obtain
\begin{align*}
        |G_{p, \bm{t}, K}|&\leq {np+n+p \choose n}\frac{|H_{p, \bm{t}, K}|}{|K|}\left(\sum_{l=0}^pt_l^{-1}\right)^{n}
    =\frac{(n!)^{p}}{(np+p)!}\sum_{\bm{j}\in\mathcal{J}^n_p}{n \choose \bj}\bt^{n-\bj} |K|^p,
        \end{align*}
which completes the proof.
     \end{proof}
\end{lemma}

\begin{thm}\label{T4a}
     Let $p\in\mathbb{N}$. For any $K\in\mathcal{K}(\mathbb{R}^n)$ and $\bm{t}\in\Theta_p$, one has
     \begin{align}\label{kokl}
                \sum_{\bm{j}\in\mathcal{J}^n_p}\bt^{n-\bj}W_{p, K}(\bm{j})\leq\frac{(n!)^p}{(np)!}\sum_{\bm{j}\in\mathcal{J}^n_p}{n \choose \bj}\bt^{n-\bj}|K|^p.
            \end{align}
        If $\dim K=n$ and equality holds for some $\bm{t}\in\relint(\Theta_p)$, then $K$ is a simplex.
        \end{thm}

     \begin{proof}
      First, the inequality \eqref{kokl} follows at once by combining \eqref{eq:GptK} with \eqref{eq:L3a}.
     
   Second, assume that $\dim(K)=n$ and that equality holds in \eqref{kokl} for some $\bt\in\relint(\Theta_p)$. Then we have equality in \eqref{eq:L3a} and the proof of Lemma \ref{L3a} together with Lemma \ref{L2} in turn imply that all non-empty sections of the body $H_{p,\bt, K}$ by the affine subspaces of the form $E_{\bt}+(0, 0, \ba)$, where $\ba=(a_1, \dots, a_p)\in(\mathbb{R}^{n})^p$, are homothetic. Denote $\rho_{\bt}:=\sum_{l=0}^pt_l^{-1}$. Since
   \begin{align*}
                H_{p,\bm{t}, K}\cap\big(E_{\bm{t}}+(0, 0, \ba)\big)&=\left\{\left(\hat\theta_{\bm{t}}, \rho_{\bt}^{-1}x, \ba\right)\mid x\in K,\rho_{\bt}\ba \in-\Delta_p(x)+K^p\right\}\\&=\left\{(\hat \theta_{\bm{t}}, y, \ba) \mid y\in\rho_{\bt}^{-1}K\cap\bigcap_{l=1}^p\left(\rho_{\bt}^{-1}K-a_l\right)\right\}\\
                &=\{\theta_{\bm{t}}\}\times\left(\rho_{\bt}^{-1}K\cap\bigcap_{l=1}^p\left(\rho_{\bt}^{-1}K-a_l\right)\right)\times\{\ba\},
            \end{align*}
           it follow that the bodies 
            $$\rho_{\bt}^{-1}K\cap\bigcap_{l=1}^p\left(\rho_{\bt}^{-1}K-a_l\right),  \quad \ba\in \rho_{\bt}^{-1}(D_pK),$$
           are all mutually homothetic. By Corollary \ref{cor:simplex}, $K$ thus has to be a simplex.
    \end{proof}
    
    Observe that Conjecture \ref{KK2}, if true, clearly implies Theorem \ref{T4a}. It is also easy to see that equality holds in \eqref{kokl} for any $K\in\calK(\RR^n)$ if $\bt$ lies on the relative boundary of $\Theta_p$. Finally, integrating \eqref{kokl} with respect to $\bm{t}$ over the set $\Theta_p$ at once yields the following corollary, analogous to Corollary \ref{cor:average}:
    \begin{cor}
\label{cor:average_mixed}
    Let $p\in\mathbb{N}$. For any $K\in\mathcal{K}(\mathbb{R}^n)$, one has
 \begin{align*}
            \sum_{\bm{j}\in\mathcal{J}^n_p}{np \choose n-\bj}^{-1}W_{p, K}(\bm{j})\leq\frac{(n!)^p}{(np)!}\sum_{\bm{j}\in\mathcal{J}^n_p}{np \choose n-\bj}^{-1}{n \choose \bj}|K|^{p}.
            \end{align*}
    \end{cor}

\subsection{Weighted inequalities III}

\label{ss:weightedIII}

Here we prove  a different weighted inequality for mixed volumes of $-\Delta_pK$ and $\iota_1K, \dots, \iota_pK$, analogous to Theorem \ref{T5}. Let us start by introducing the corresponding higher-order generalization of unbalanced difference body, since it will play a role in the proof. Let $p\in\NN$. For any $\bt=(t_0,\dots,t_p)\in\Theta_p$ and $K\in\calK(\RR^n)$ we define
$$
D_p^{\bt}K:=-t_0\Delta_pK+\sum_{l=1}^pt_l\iota_lK.
$$
Observe that, for $p=1$, $D_1^{(t_0,1-t_0)}K=D^{t_0}K$ agrees with the unbalanced difference body defined by Artstein-Avidan and Putterman \cite{ArtsteinPutterman25}. Moire details about the body $D_p^{\bt}K$ will be given and the corresponding Schneider--Rogers--Shephard type inequality will be conjectured in \S\ref{ss:diffII}.

\begin{lemma}
\label{lem:ftK}
Let $p\in\NN$, $K\in\calK(\RR^n)$ with $\dim K=n$, and let $\bt\in\relint(\Theta_p)$ be such that $t_0=\min\{t_0, \dots, t_p\}$. The function
$$f_{\bt, K}(\bx):=|t_0K|^{-1}\,\big|(t_0K)\cap(t_1K-x_1)\cap\cdots\cap (t_pK-x_p)\big|, \quad\bx=(x_1,\dots,x_p)\in(\RR^n)^p,$$ 
is continuous and satisfies the following properties:
\begin{enuma}
\item $\int_{\RR^{np}}f_{\bt,K}(\bx)\d\bx=|K|^p\prod_{l=1}^pt_l^n$.
\item $\supp f_{\bt, K}=D_p^{\bt}K$.
\item $f_{\bt, K}(\RR^{np})=[0,1]$. Moreover, $f_{\bt, K}(\bm{x})=0$ if and only if $\bx\notin\inter (D_p^{\bt} K)$, and $f_{\bm{t}, K}(\bm{x})=1$ if and only if $\bm{x}\in\sum_{l=1}^p(t_l-t_0)\iota_lK.$
\item For any $\by,\bz\in D_p^{\bt}K$ and $s\in[0,1]$, it holds that
$$
f_{\bm{t}, K}\big((1-s)\bm{y} + s\bm{z}\big)^{\frac{1}{n}} \geq (1-s)f_{\bm{t},K}(\bm{y})^{\frac{1}{n}}+sf_{\bm{t}, K}(\bm{z})^{\frac{1}{n}}.
$$ 
\end{enuma}
\end{lemma}

\begin{proof}
First, continuity is clear. Second, (a) follows easily from Fubini's theorem since
\begin{align*}
   \int_{\RR^{np}}f_{\bt,K}(\bx)\d\bx&= |t_0K|^{-1} \int_{\RR^{np}}\big|(t_0K)\cap(t_1K-x_1)\cap\cdots\cap (t_pK-x_p)\big|\d\bx\\
   &=  |t_0K|^{-1}  \int_{\mathbb{R}^{np}}\int_{\mathbb{R}^n}\chi_{t_0K}(x_0)\prod_{l=1}^p\chi_{t_lK}(x_0+x_l)\d x_0\d \bx\\
   &=  |t_0K|^{-1}  \int_{\mathbb{R}^{n}}\chi_{t_0K}(x_0)\left(\prod_{l=1}^p\int_{\mathbb{R}^n}\chi_{t_lK}(x_0+x_l)\d x_l\right)\d x_0\\
    &=  |t_0K|^{-1}\left(\prod_{l=1}^p t_l^n\right)|K|^p \int_{\mathbb{R}^n} \chi_{t_0K}(x_0)\d x_0\\
    & =|K|^p\prod_{l=1}^pt_l^n.
    \end{align*}
    Third, it is straightforward to verify, cf. proof of \cite[Proposition 3.1]{K25}, that
$$
D_p^{\bt}K=\left\{\bx\in(\RR^n)^p\mid t_0K\cap(t_1K-x_1)\cap\cdots\cap(t_pK-x_p)\neq\emptyset\right\}.
$$
This immediately implies (b) and the second part of (c). Since clearly $f_{\bt, K}(\RR^{np})\subset[0,1]$, the first part of (c) follows by continuity. Finally, to prove (d), notice first that we have
    \begin{align*}
           & (t_0K)\cap\bigcap_{l=1}^p\Big(t_lK-\big((1-s)y_l+sz_l\big)\Big)\\ 
           &\quad=\big((1-s)t_0K+st_0K\big)\cap\bigcap_{l=1}^p\big((1-s)(t_lK-y_l)+s(t_lK-z_l)\big)\\
           &\quad\supseteq (1-s)\big((t_0K)\cap\bigcap_{l=1}^p(t_lK-y_l)\big)+s\big((t_0K)\cap\bigcap_{l=1}^p(t_lK-z_l)\big).
        \end{align*}
    This and the Brunn--Minkowski inequality (Theorem \ref{BM}) readily imply the claim.
\end{proof}

\begin{thm}\label{T5a}
Let $p\in\mathbb{N}$. For any $K\in\mathcal{K}(\mathbb{R}^n)$ and $\bt\in\Theta_p$, one has
\begin{align}
\label{f000}
    \sum_{\bm{j}\in\mathcal{J}^n_p} q_{n, p, \bm{j}}(\bm{t})W_{p,  K}(\bm{j})\leq \bt^n|K|^p,
\end{align}
where, denoting $t_{\min}:=\min\{t_0, \dots, t_p\}$, we set
$$
q_{n, p, \bm{j}}(\bm{t}) :=n(t_{\min})^n{np \choose n-\bj}\int_{0}^1 (1-s)^{n-1} \prod_{l=0}^p\big(t_l-(1-s)t_{\min}\big)^{n-j_l}\d s.
$$
If $\dim(K)=n$ and equality in \eqref{f000} holds for some $\bm{t}\in \relint(\Theta_p)$, then $K$ is a simplex.
\end{thm}

\begin{proof}
It suffices to prove \eqref{f000} for $K$ with $\dim(K)=n$ and for $\bm{t}\in\relint(\Theta_p)$. Without loss of generality we may also assume that $t_{\min}=t_0$. Then one has the decomposition
\begin{align}
\label{eq:decompD}
D_p^{\bt}K=K_{\bt}+K'_{\bt},
\end{align}
where $K_{\bm{t}}:=\sum_{l=1}^p(t_l-t_0)\iota_lK$ and $K'_{\bm{t}}:= t_0D_pK$. For any $\bx\in D_p^{\bt}K$ we set
$$
s_{\bm{x}}:=\min\big\{s\in[0, 1] \mid \bm{x}=(1-s)\bm{y}+s\bm{z}\text{ for some } \bm{y}\in K_{\bm{t}} \text{ and } \bm{z}\in D_p^{\bt}K\big\}.
$$
It holds
\begin{equation}
\label{f12b}
    s_{\bm{x}}=d_{K'_{\bm{t}}}(\bm{x}, K_{\bm{t}}).
\end{equation}
Indeed, set $\rho\coloneqq d_{K'_{\bm{t}}}(\bm{x}, K_{\bm{t}})=\min\{r\geq0 : \bm{x}\in K_{\bm{t}}+rK'_{\bm{t}}\}$ and observe that \eqref{eq:decompD} implies that $\rho\in[0,1]$. Then $\bm{x}=\bm{y}+\rho\bm{z}$ for some $\bm{y}\in K_{\bm{t}}$ and $\bm{z}\in K'_{\bm{t}}$. Consequently,
$$
\bm{x}=(1-\rho)\bm{y}+r(\bm{y}+\bm{z})\in (1-\rho)K_{\bm{t}}+\rho(K_{\bm{t}}+K'_{\bm{t}}),
$$
which in turn implies $s_{\bm{x}}\leq d_{K'_{\bm{t}}}(\bm{x}, K_{\bm{t}}).$ Conversely, for some $\bm{y},\bz_1\in K_{\bm{t}}$ and $\bm{z}_2\in K'_{\bm{t}}$, one has
$$
\bm{x}=(1-s_{\bm{x}})\bm{y}+s_{\bm{x}}(\bz_1+\bz_2)=\big((1-s_{\bm{x}})\bm{y}+s_{\bm{x}}\bm{z}_1\big)+s_{\bm{x}}\bm{z}_2\in K_{\bm{t}}+s_{\bm{x}}K'_{\bm{t}},
$$
from which we deduce that $s_{\bm{x}}\geq d_{K'_{\bm{t}}}(\bm{x}, K_{\bm{t}})$.

Consider the function $f_{\bt,K}:\RR^{np}\to[0,1]$ from Lemma \ref{lem:ftK}. Items (c) and (d) of Lemma \ref{lem:ftK} imply that for any $\bx\in D_p^{\bt}K$ one has $
f_{\bt,K}(\bx)\geq(1-s_{\bx})^n$. Using \eqref{f12b} and Lemma \ref{lem:ftK}(b), this can be written as
\begin{equation}\label{f7b}
    f_{\bm{t}, K}(\bm{x})\geq \big(1-d_{K'_{\bm{t}}}(\bm{x}, K_{\bm{t}})\big)^n \chi_{D_p^{\bt}K}(\bm{x}).
\end{equation}
Consequently, using also Lemma \ref{lem:ftK}(a), Fubini's theorem, and \eqref{eq:decompD}, we can write
\begin{align*}
\bt^n|K|^p&=t_0^n\int_{\RR^{np}}f_{\bt,K}(\bx)\d\bx\\
&\geq t_0^n\int_{D_p^{\bt}K}\big(1-d_{K'_{\bm{t}}}(\bm{x}, K_{\bm{t}})\big)^n\d\bx\\
&=t_0^n\int_0^1\left|\big\{\bx\in D_p^{\bt}K \mid \big(1-d_{K'_{\bm{t}}}(\bm{x}, K_{\bm{t}})\big)^n\geq r\big\}\right|\d r\\
&=t_0^n\int_0^1\left|\big\{\bm{x}\in K_{\bm{t}}+K'_{\bm{t}} \mid d_{K'_{\bm{t}}}(\bm{x}, K_{\bm{t}})\leq 1-r^{\frac{1}{n}}\big\}\right|\d r\\
&=nt_0^n\int_0^1(1-s)^{n-1}\left|\big\{\bm{x}\in K_{\bm{t}}+K'_{\bm{t}} \mid d_{K'_{\bm{t}}}(\bm{x}, K_{\bm{t}})\leq s\big\}\right|\d s\\
&=nt_0^n\int_0^1(1-s)^{n-1}\left|K_{\bm{t}}+sK'_{\bm{t}}\right|\d s\\
&=nt_0^n\int_0^1(1-s)^{n-1}\left|-st_0\Delta_pK+\sum_{l=1}^p(t_l-t_0+st_0)\iota_lK\right|\d s\\
&=\sum_{\bm{j}\in\mathcal{J}^n_p}q_{n,  p,  \bm{j}}(\bm{t})W_{p, K}(\bm{j}).
\end{align*}
This proves \eqref{f000}.

Suppose that equality holds  in \eqref{f000} for $K\in\calK(\RR^n)$ with $\dim(K)=n$ for some $\bm{t}\in\relint(\Theta_p)$. By symmetry, we may again assume that $t_{\min}=t_0$. By continuity, we have
\begin{equation}\label{f14a}
    f_{\bm{t}, K}(\bm{x})=\big(1-d_{K'_{\bm{t}}}(\bm{x}, K_{\bm{t}})\big)^n=(1-s_{\bx})^n
\end{equation}
for all $\bx\in D_p^{\bt}K$.

Fix arbitrary $\bx \in \inter(D_p^{\bt}K)\setminus K_{\bt}$. Note that $K_{\bt}\subset\inter (D_p^{\bt}K)$ since $D_p^{\bt}K=K_{\bt}+K'_{\bt}$ and $0\in\inter (K'_{\bt})$. Then we have $s_{\bm{x}}\in (0, 1)$ and there exist $\bm{y}\in K_{\bm{t}}$ and $\bm{z}\in \partial K'_{\bm{t}}$ such that $\bm{x}=\bm{y}+s_{\bm{x}}\bm{z}$. For $r\in [s_{\bm{x}}, 1)$, set  $ \bm{x}_r:= \bm{y}+r\bm{z}$. We claim that
\begin{align}
\label{eq:d=r}
d_{K'_{\bm{t}}}(\bm{x}_r, K_{\bm{t}})=r
\end{align}
for all $r\in [s_{\bm{x}}, 1)$. Indeed, on the one hand, $\bm{x}_r\in K_{\bm{t}}+rK'_{\bm{t}}$ implies $d_{K'_{\bm{t}}}(\bm{x}_r, K_{\bm{t}})\leq r$. On the other hand, by \eqref{f12b}, $\bm{x}\in\partial (K_{\bm{t}}+s_{\bm{x}}K'_{\bm{t}})$ and hence there exists $\bm{w}\in\RR^{pn}\setminus\{0\}$ with
$$
\bm{x}\cdot\bm{w}=h_{K_{\bm{t}}+s_{\bm{x}}K'_{\bm{t}}}(\bm{w})= h_{K_{\bm{t}}}(\bm{w})+s_{\bm{x}}h_{K'_{\bm{t}}}(\bm{w}).
$$
Since
$$
\bm{x}\cdot\bm{w}=\bm{y}\cdot\bm{w}+s_{\bm{x}}\bm{z}\cdot\bm{w}\leq h_{K_{\bm{t}}}(\bm{w})+s_{\bm{x}}h_{K'_{\bm{t}}}(\bm{w}),
$$
we also have $h_{K_{\bm{t}}}(\bm{w})=\bm{y}\cdot\bm{w}$ and $h_{K'_{\bm{t}}}(\bm{w})=\bm{z}\cdot\bm{w}>0$. Suppose $d_{K'_{\bm{t}}}(\bm{x}_r, K_{\bm{t}})< r$. Then there are $r_0\in[0,r)$, $\bm{u}\in K_{\bm{t}}$, and $\bm{v}\in K'_{\bm{t}}$ such that $\bm{x}_r=\bm{u}+r_0\bm{v}$ and hence
$$
        \bm{x}_r\cdot\bm{w}=\bm{u}\cdot\bm{w}+r_0\bm{v}\cdot\bm{w}\leq h_{K_{\bm{t}}}(\bm{w})+r_0h_{K'_{\bm{t}}}(\bm{w})= \bm{y}\cdot\bm{w}+r_0\bm{z}\cdot\bm{w}<\bm{y}\cdot\bm{w}+r\bm{z}\cdot\bm{w}= \bm{x}_r\cdot\bm{w},
    $$
a contradiction.

Observe that $\bm{x}=(1-r^{-1}s_{\bm{x}})\bm{y}+r^{-1}s_{\bm{x}}\bm{x}_r$. Thus, using \eqref{f14a}, Lemma \ref{lem:ftK}, and \eqref{eq:d=r}, we have
$$
   1-s_{\bm{x}} =f_{\bm{t}, K}(\bm{x})^{\frac{1}{n}}\geq(1-r^{-1}s_{\bm{x}})f_{\bm{t}, K}(\bm{y})^{\frac{1}{n}}+r^{-1}s_{\bm{x}}f_{\bm{t},  K}(\bm{x}_r)^{\frac{1}{n}}=1-s_{\bm{x}}.
$$
This in particular yields $f_{\bm{t}, K}(\bm{x})^{\frac{1}{n}}=(1-r^{-1}s_{\bm{x}})f_{\bm{t}, K}(\bm{y})^{\frac{1}{n}}+r^{-1}s_{\bm{x}}f_{\bm{t},  K}(\bm{x}_r)^{\frac{1}{n}}$ for some $r\in(s_{\bx},1)$. Consequently, the proof of Lemma \ref{lem:ftK}(d) together with Theorem \ref{BM} implies that the convex bodies $(t_0K)\cap\bigcap_{l=1}^p(t_lK-y_l)$ and $(t_0K)\cap\bigcap_{l=1}^p(t_lK-x_l)$ are homothetic. Moreover, since $\bm{y}\in K_{\bm{t}}$, we have $(t_0K)\cap\bigcap_{l=1}^p(t_lK-y_l)=t_0K$ and we conclude that $K$ is homothetic to $(t_0K)\cap\bigcap_{l=1}^p(t_lK-x_l)$ for any $\bm{x}\in\inter (D_p^{\bt}K)\setminus K_{\bm{t}}$.

We claim that $(t_0K)\cap (t_1K-x)$ is homothetic to $K$ for any $x\in\inter(t_1K-t_0K)$. First, if $x\in(t_1-t_0)K\subset\inter(t_1K-t_0K)$, then $t_0K+x\subset t_0K+(t_1-t_0)K=t_1K$ and therefore $(t_0K)\cap (t_1K-x)=t_0 K$ is homothetic to $K$. Second, assume $x\in\inter(t_1K-t_0K)\setminus (t_1-t_0)K$. Choose $y, z\in \inter(K)$ such that $x=t_1y-t_0z$ and set
\begin{align*}
x_l:=\begin{cases}x,&\text{for }l=1,\\(t_l-t_0)z,&\text{for } l=2,\dots,p.\end{cases}
\end{align*}
Then $x_l+t_0z\in t_l\inter(K)$ for all $l=1,\dots,p$, and consequently $\bm{x}:= (x_1, \dots, x_p)\in\mathbb{R}^{np}$ satisfies
\begin{align*}
        \bm{x}=-t_0\Delta_p(z)+\big(\bm{x}+t_0\Delta_p(z)\big)\in\inter\big(-t_0\Delta_p(z)+\sum_{i=1}^pt_i\iota_iK\big)\subset\inter (D_p^{\bt}K).
    \end{align*}
    At the same time, $\bm{x}\notin K_{\bm{t}}$ because $x\notin (t_l-t_0)K$. Hence $(t_0K)\cap\bigcap_{l=1}^p(t_lK-x_l)$ is homothetic to $K$. Finally, since for $l\geq2$ we have
    $$
    t_0K+x_l=t_0K+(t_l-t_0)z\subset t_0K+(t_l-t_0)K= t_lK,
    $$
    or, equivalently, $t_0K\subset t_lK-x_l$, $K$ is in fact homothetic to $(t_0K)\cap(t_1K-x_1)$ for arbitrary $x\in\inter(t_1K-t_0K)$. Corollary \ref{cor:simplex} then implies that $K$ is a simplex. 
\end{proof}

We will now show that Theorem \ref{T5} follows as a consequence of Theorem \ref{T5a}.

\begin{proof}[Proof of Theorem \ref{T5}]
We may assume $t\in(0,1]$. Then it is straightforward to verify that \eqref{f00} follows from \eqref{f000} if we set $\bm{t}\coloneqq \frac{1}{p+t}(t,1, \dots, 1)\in\Theta_p$. If equality holds in \eqref{f00} for some $t\in(0,1)$, then we have equality in  \eqref{f000} for $\bm{t}= \frac{1}{p+t}(t,1, \dots, 1)\in\relint(\Theta_p)$, and hence $K$ is a simplex.
\end{proof}

\subsection{Higher-order unbalanced difference body II}

\label{ss:diffII}

Recall from \S\ref{ss:weightedIII} that we define for $p\in\NN$ and $\bt=(t_0,\dots,t_p)\in\Theta_p$ the higher-order unbalanced difference body of $K\in\calK(\RR^n)$ as
$$
D_p^{\bt}K:=-t_0\Delta_pK+\sum_{l=1}^pt_l\iota_lK.
$$
Motivated by \cite[Conjecture 2]{ArtsteinPutterman25}, we conjecture the corresponding unbalanced Schneider--Rogers--Shephard type inequality, analogous to Conjecture \ref{COn}:

\begin{con}\label{COn2}
   Let $p\in\mathbb{N}$. For any $K\in\mathcal{K}(\mathbb{R}^n)$ with $\dim(K)=n$ and any $\bt\in\Theta_p$, one has
     \begin{equation}\label{UUU}
        \frac{|D_p^{\bt}K|}{|K|^p}\leq \frac{|D_p^{\bt}S_n|}{|S_n|^p}
    \end{equation}
     where $S_n\in\calK(\RR^n)$ is an $n$-simplex. If equality holds for some $t\in\relint(\Theta_p)$, then $K$ is a simplex.
\end{con}

Observe that
$$
|D_p^{\bt}K|=\sum_{\bm{j}\in\mathcal{J}^n_p}{np \choose n-\bj}\bt^{n-\bj} W_{p, K}(\bm{j}).
$$
In particular, if $K=S_n$ is an $n$-simplex, we have
\begin{equation}\label{f0a}
    |D_p^{\bt}S_n|=\sum_{\bm{j}\in\mathcal{J}^n_p}{n \choose j_0}\cdots{n \choose j_p}\bt^{n-\bj}|S_n|^p.
\end{equation}
The inequality \eqref{UUU} can thus be equivalently stated as
$$
\sum_{\bm{j}\in\mathcal{J}^n_p}{np \choose n-\bj}\bt^{n-\bj} W_{p, K}(\bm{j})\leq\sum_{\bm{j}\in\mathcal{J}^n_p}{n \choose j_0}\cdots{n \choose j_p}\bt^{n-\bj}|K|^p,\quad \bt\in\Theta_p.
$$       
In this formulation it is obvious that Conjecture \ref{COn2} would follow from Conjecture \ref{KK2}. Moreover, by the same considerations as in the proof of Theorem \ref{T5}, it is easily seen that Conjecture \ref{COn2} implies Conjecture \ref{COn}. Let us show that Conjecture \ref{COn2}, if true, directly implies the weighted inequalities proven in Theorems \ref{T4a} and \ref{T5a}.

\begin{prop}
\label{pro:impliesT4a}
    Conjecture \ref{COn2} implies Theorem \ref{T4a}.
\end{prop}
\begin{proof}
Fix $p\in\mathbb{N}$ and $\bt\in\relint(\Theta_p)$, and for $\bs\in\Theta_p$ set
$$
\alpha_{\bm{t}}(\bm{s}):= \sum_{l=0}^ps_lt_l\qquad\text{and}\qquad\mu_{\bm{t}}(\bm{s})\coloneqq\frac{1}{\alpha_{\bm{t}}(\bm{s})}(s_0t_0, \dots, s_pt_p).
$$
Then $\mu_{\bt}(\bs)\in\Theta_p$ and we can write
\begin{align}
\label{eq:volGptK}
        |G_{p,\bm{t}, K}|=\int_{\hat\Theta_p}\left|-s_0t_0\Delta_pK+\sum_{l=1}^ps_lt_l\iota_lK\right|\d \hat\bs=\int_{\hat\Theta_p}\alpha_{\bm{t}}(\bm{s})^{np}|D_p^{\mu_{\bm{t}}(\bm{s})}K|\d \hat\bs.
    \end{align}
    Using \eqref{eq:volGptK}, \eqref{f0a}, and assuming $K$ satisfies \eqref{UUU}, we have
    \begin{align*}
        |G_{p,\bm{t}, K}|&\leq \frac{|K|^p}{|S_n|^p}\int_{\hat\Theta_p}\alpha_{\bm{t}}(\bm{s})^{np}|D_p^{\mu_{\bm{t}}(\bm{s})}S_n|\d \hat\bs \\
        &=|K|^p\sum_{\bm{j}\in\mathcal{J}^n_p}{n \choose j_0}\cdots{n \choose j_p}\int_{\hat\Theta_p} \alpha_{\bm{t}}(\bm{s})^{np}\mu_{\bm{t}}(\bm{s})^{n-\bj}\d\hat\bs\\
        &=|K|^p\sum_{\bm{j}\in\mathcal{J}^n_p}{n \choose j_0}\cdots{n \choose j_p}\bt^{n-\bj}\int_{\hat\Theta_p} \bs^{n-\bj}\d\hat\bs\\
        &=\frac{(n!)^p}{(np+p)!}\sum_{\bm{j}\in\mathcal{J}^n_p}{n \choose \bj}\bt^{n-\bj}|K|^p.
\end{align*}
Together with \eqref{eq:GptK}, this immediately implies \eqref{kokl}. For $\bt$ on the relative boundary of $\Theta_p$, the inequality is trivial.

Let $S_n\in\calK(\RR^n)$ be an $n$-simplex and assume that $|K|=|S_n|>0$. If equality holds in \eqref{kokl} for some $t\in\relint(\Theta_p)$, then
$$
\left|D_p^{\mu_{\bt(\bs)}}K\right|=\left|D_p^{\mu_{\bt(\bs)}}S_n\right|
$$
for some $\bs\in\relint(\Theta_p)$ and Conjecture \ref{COn2}, if true, implies that $K$ is an $n$-simplex.
\end{proof}

\begin{prop}
\label{pro:impliesT5a}
    Conjecture \ref{COn2} implies Theorem \ref{T5a}.
\end{prop}
\begin{proof}
Fix $p\in\NN$ and $\bt\in\Theta_p$ and keep the notation from Theorem \ref{T5a}. Without loss of generality we may again assume $t_{\min}=t_0$. Then the very last step in the proof of \eqref{f000} shows
\begin{align*}
    \sum_{\bm{j}\in\mathcal{J}^n_p}q_{n,  p,  \bm{j}}(\bm{t})W_{p, K}(\bm{j})=nt_0^n\int\limits_{0}^1(1-s)^{n-1}\beta_{\bm{t}}(s)^{np}|D_p^{\nu_{\bm{t}}(s)}K|\d s,
\end{align*}
where for $s\in[0,1]$ we set
$$
\beta_{\bm{t}}(s):= \sum_{l=0}^p (t_l-t_0+st_0)\quad\text{and}\quad\nu_{\bm{t}}(s):=\frac{1}{\beta_{\bm{t}}(s)}(st_0,t_1-t_0+s t_0, \dots, t_p-t_0+st_0).
$$
The rest of the proof is completely analogous to the proof of Proposition \ref{pro:impliesT4a}, using the relation
$$\frac{(n!)^p}{(np)!}\sum_{\bj\in\calJ_p^n}{n \choose \bj} q_{n, p,\bj}(\bt)=\bt^n$$
which is readily verified by a direct computation.
\end{proof}

\section{Inequalities for the unbalanced join of  $-\Delta_pK$ and $K^p$}
\label{s:join1}

In this section, we will prove an upper bound on the volume of the unbalanced join of the bodies $-\Delta_pK$ and $K^p$ for $K\in\calK(\RR^n)$, thus proving Theorem \ref{thm:join}. To this end, for $p\in\NN$, $t\in[0,1]$, and $K\in\calK(\RR^n)$ we denote
$$
Q_{p, t, K}:=(1-t)K^p \vee -t\Delta_pK.
$$
Notice that for $E=\{0\}\times\RR^{np}$, we have $P_E(C_{p,t,K})= Q_{p,t,K}$, see \eqref{eq:defCptK}. The following reduction lemma will crucial:

\begin{lemma}
\label{lem:reduction}
    Let $p\in\mathbb{N}$ and $t\in[0,1]$. If $S_n\in\calK(\RR^n)$ is an $n$-simplex, then we have
    \begin{align*}
            |Q_{p,t, S_n}|=(1-t)^{n(p-1)}|S_n|^{p-1}|Q_{1, t, S_n}|.
    \end{align*}
    \end{lemma}

    \begin{proof}
    Clearly we may assume that $p\geq 2$ and $t\in(0,1)$. Consider the linear map $F:\RR^{np}\to\RR^{np}$ given by
    $$
    F(x_1, \dots, x_p)\coloneqq (x_1-x_p, \dots, x_{p-1}-x_p, x_p).
    $$ 
Observe that $\det(F)=1$, $F(\Delta_pS_n)=\{0\}\times S_n$, and $$
F(S_n^p)=\left\{(\bw, z)\in(\mathbb{R}^n)^{p-1}\times\mathbb{R}^n : z\in S_n\cap\bigcap_{i=1}^{p-1}(S_n-w_i), \bw=(w_1, \dots, w_{p-1})\right\}.
$$
Consider further the functions $\rho(\bw)$ and $y(\bw)$ given by  \eqref{eq:rhoy}. According to \eqref{eq:Sncap}, the fibre
$$
F(S_n^p)_{\bw}:=\{z\in\mathbb{R}^n : (\bw, z)\in F(S_n^p)\}=S_n\cap\bigcap_{i=1}^{p-1}(S_n-w_i),\quad \bw\in(\mathbb{R}^n)^{p-1},
$$
is non-empty if and only if $\rho(\bw)\geq0$ in which case it is given by
\begin{align}
\label{eq:fibre}
F(S_n^p)_{\bw}=y(\bw)+\rho(\bw)S_n.
\end{align}
Recall also that for $s\geq 1-\rho(\bw)$ we have $y(s^{-1}\bw)=s^{-1}y(\bw)$ and $\rho(s^{-1}\bw)=1-s^{-1}\big(1-\rho(\bw)\big)$, while for $s< 1-\rho(\bw)$ we have $F(S_n^p)_{s^{-1}\bw}=\emptyset$.

Denote $\rho_+(\bw):=\max\{0, \rho(\bw)\}$. We claim that
\begin{align}\label{WEW}
            \int_{(\RR^n)^{p-1}}\rho_+(\bw)^n\d\bw=|S_n|^{p-1}.
        \end{align}
On the one hand, using \eqref{eq:fibre} and the fact that $\rho_+(\bw)=0$ if $\bw\notin D_{p-1}S_n$, we have
\begin{align*}
\int_{D_{p-1}S_n}|F(S_n^p)_{\bw}|\d\bw  =|S_n|\int_{D_{p-1}S_n}\rho_+(\bw)^n\d\bw= |S_n|\int_{(\RR^n)^{p-1}}\rho_+(\bw)^n\d\bw.
\end{align*}
On the other hand, by Fubini's theorem, the first integral equals
        \begin{align*}
        \int\limits_{D_{p-1}S_n}\int_{\mathbb{R}^n}\chi_{S_{n}}(z)\prod_{i=1}^{p-1}\chi_{S_{n}-w_i}(z)\d z\d\bw=\int_{\mathbb{R}^n}\chi_{S_{n}}(z)\left(\prod_{i=1}^{p-1}\int\limits_{\mathbb{R}^n}\chi_{S_{n}}(z+w_i)\d w_i\right)\d z=|S_n|^p
            \end{align*}
and \eqref{WEW} follows.

Now consider the fibres of  $F(Q_{p, t, S_n})$. We will show that for any $\bw\in(\mathbb{R}^n)^{p-1}$ with $\bw\neq0$, 
\begin{align}
\label{WWWWW}
            F(Q_{p, t, S_n})_{\bw}=\bigcup_{r\in(0, 1]}r(1-t)F(S_n^p)_{r^{-1}(1-t)^{-1}\bw}+(1-r)(-tS_n).
        \end{align} 
        If $(w, z)\in F(Q_{p, t, S_n})$, then there exist $r_0\in[0, 1]$, $(\bw_0, z_0)\in F(S_n^p)$, and $y_0\in S_n$ such that \begin{align*}
                (\bw, z)=r_0(1-t)(w_0, z_0)+(1-r_0)(0, -ty_0)=\big(r_0(1-t)w_0, r_0(1-t)z_0+(1-r_0)(-ty_0)\big).
            \end{align*}
Since $\bw\neq0$, we have $r_0\in (0, 1]$, $\bw_0=r_0^{-1}(1-t)^{-1}\bw$, and  $z_0\in F(S_n^p)_{\bw_0}=F(S_n^p)_{r_0^{-1}(1-t)^{-1}\bw}$. Consequently,
$$
z=r_0(1-t)z_0+(1-r_0)(-ty_0)\in r_0(1-t)F(S_n^p)_{r_0^{-1}(1-t)^{-1}\bw}+(1-r_0)(-tS_n).
$$
Conversely, for any $r\in(0, 1]$, $z\in F(S_n^p)_{r^{-1}(1-t)^{-1}\bw}$, and $y\in S_n$, we have 
\begin{align*}
                r(1-t)\big(r^{-1}(1-t)^{-1}\bw, z\big)+(1-r)(0, -ty)=\big(\bw, r(1-t)z-(1-r)ty\big),
            \end{align*}
            and hence $r(1-t)z+(1-r)(-ty)\in F(Q_{p, t, S_n})_{\bw}$. This completes the proof of \eqref{WWWWW}.
            
Let $0\neq\bw\in(\RR^n)^{p-1}$ and set $\bw_t:=(1-t)^{-1}\bw$. We next claim that
\begin{align}            
 \label{WEM}
          |F(Q_{p, t, S_n})_{\bw}|=\rho_+(\bw_t)^n|Q_{1, t, S_n}|.
             \end{align}           
      Indeed, using \eqref{eq:fibre}, \eqref{WWWWW}, and the fact that $F(S_n^p)_{r^{-1}\bw_t}=\emptyset$ for $r<1-\rho(\bw_t)$, we get
             \begin{align*}
                 F(Q_{p, t, S_n})_{\bw}&=\bigcup_{r\in(0,1]}r(1-t)F(S_n^p)_{r^{-1}\bw_t}+(1-r)(-tS_n)\\
                 &=\bigcup_{r\in(0, 1]}r(1-t)\big(y(r^{-1}\bw_t)+\rho(r^{-1}\bw_t)S_n\big) +(1-r)(-tS_n)\\
                 &=\bigcup_{r\in[1-\rho(\bw_t), 1]}(1-t)y(\bw_t)+\big(r-1+\rho(\bw_t)\big)(1-t)S_n +(1-r)(-tS_n).
             \end{align*} 
         If $\rho(\bw_t)=0$, we have $|F(Q_{p, t, S_n})_{\bw}|=0$ since $F(Q_{p, t, S_n})_{\bw}$ is a singleton. If $\rho(\bw_t)>0$, using the substitution $r^{\prime}\coloneqq \rho(\bw_t)^{-1}\big(r-1+\rho(\bw_t)\big)$ for $r\in[1-\rho(\bw_t),\, 1]$, we obtain
             \begin{align*}
                 F(Q_{p, t, S_n})_{\bw}=(1-t)y(\bw_t)+\rho(\bw_t) Q_{1, t, S_n}
             \end{align*}
             and \eqref{WEM} follows.
     
 Now, Fubini's theorem, \eqref{WEW}, and \eqref{WEM} yield 
 \begin{align*}
             |Q_{p, t, S_n}|&= |F(Q_{p, t, S_n})|\\
             &=\int_{(\mathbb{R}^n)^{p-1}}|F(Q_{p, t, S_n})_{\bw}|\d\bw \\
             &=(1-t)^{n(p-1)}|Q_{1, t, S_n}|\int_{(\mathbb{R}^n)^{p-1}}\rho_+(\bw)^n\d\bw\\
             &=(1-t)^{n(p-1)}|S_n|^{p-1}|Q_{1, t, S_n}|,
            \end{align*}
            as claimed.
    \end{proof}

\begin{lemma}
\label{LQ}
   Let $p\in\mathbb{N}$, $t\in[0,1]$, and let $S_n\in\calK(\RR^n)$ be an $n$-simplex.
   \begin{enuma}
   \item If $S_n$ is centered, then
        \begin{align*}
             |Q_{p,t, S_n}|={n \choose k} (1-t)^{np-k}t^k|S_n|^p,
        \end{align*}
\item If $0$ is a vertex of $S_n$, then
        \begin{align*}
             |Q_{p,t, S_n}|=(1-t)^{n(p-1)}|S_n|^p.
        \end{align*}
   \end{enuma}
\end{lemma}

\begin{proof}
The case $p=1$ was proven in \cite[Lemma 2.1 and \S5.4]{ArtsteinETAL15}. The general case then follows at once from Lemma \ref{lem:reduction}.
\end{proof}

Now we can proceed to the proof of our main result.

    \begin{proof}[Proof of Theorem \ref{thm:join}]
Fix $t\in[0,1]$ and $K\in\calK(\RR^n)$ with $0\in K$. Consider the subspace $E:=\RR\times\{0\}\subset\RR\times\RR^{np}$. Since $0\in K$, we have $ E\cap C_{p,t,K}=[0,1]$. Using this and $P_{E^\perp} C_{p,t,k}=Q_{p,t,K}$, Lemmas \ref{L2} and \ref{L3} imply
        \begin{align*}
                |Q_{p,t, K}|\leq{np+1 \choose 1}|C_{p,t, K}|=(np+1)|C_{p,t, K}|\leq(1-t)^{n(p-1)}|K|^p,
            \end{align*}
        proving \eqref{eq:join}.

Assume that $\dim K=n$ and that equality holds in \eqref{eq:join} for some $t\in(0,1)$. Then we in particular have equality in \eqref{eq:L3}. Since \eqref{eq:L3} was the only inequality used in the proof of Theorem \ref{T4}, we in turn have equality in \eqref{Eq71} and hence Theorem \ref{T4} implies that $K$ is a simplex.
            
By  Lemma \ref{lem:reduction}, any $n$-simplex satisfying equality in \eqref{eq:join} satisfies equality in  \eqref{eq:join} for $p=1$. It was shown in \cite[\S5.4]{ArtsteinETAL15} that in such a case, the simplex must have a vertex at the origin. This finishes the proof.
 \end{proof}

\begin{proof}[Proof of Corollary \ref{CORE}]
We may assume $0\in K$. Let $t\in(0,1)$. Theorem \ref{thm:join} together with the monotonicity of mixed volume applied to $(1-t)K^p\subset Q_{p, t, K}$ and $-t\Delta_pK\subset Q_{p, t, K}$ implies
            \begin{align*}
                V_{np}(K^p[np-j], -\Delta_pK[j])\leq\frac{|Q_{p,t,K}|}{t^j(1-t)^{np-j}}\leq\frac{|K|^p}{t^j(1-t)^{n-j}}.
            \end{align*}
                    Choosing $t=\frac{j}{n}$ yields the result.
    \end{proof}

To conclude this section, we will show that the corresponding F\'ary--R\'edei type conjecture (Conjecture \ref{QQ}) implies the inequality part of Conjecture \ref{KK1}.

\begin{prop}\label{JKLA}
    Conjecture \ref{QQ} implies inequality \eqref{KOTR} in Conjecture \ref{KK1}.
\end{prop}

    \begin{proof}
       Let $p\in\NN$, $1\leq j\leq n$,
       and assume that $K\in\calK(\RR^n)$ with $\dim K=n$ satisfies Conjecture \ref{QQ}. Set $           t:= \frac{j}{n+1}$ and let $x\in K$ be such that
        $$ \frac{|Q_{p, t, K-x}|}{|K|^p} \leq \frac{|Q_{p, t, S_n}|}{|S_n|^p}.$$
 Then, using Lemma \ref{LQ} together with the monotonicity and translation invariance of mixed volumes, we obtain
        \begin{align*}
                V_{np}(K^p[np-j], -\Delta_pK[j])\leq \frac{|Q_{p, t, K-x}|}{(1-t)^{np-j}t^j}\leq\frac{|Q_{p, t, S_n}|}{(1-t)^{np-j}t^j|S_n|^p}|K|^p={n \choose j}|K|^p.
        \end{align*}
    \end{proof}

\section{Inequalities for the unbalanced join of  $-\Delta_pK$ and $\iota_1K, \dots, \iota_pK$}
\label{s:join2}

In this last section, we consider the unbalanced join of the bodies  $-\Delta_pK$ and $\iota_1K, \dots, \iota_pK$ for $K\in\calK(\RR^n)$ and generalize the results of the previous section to this setting, that is, to the context of Conjecture \ref{KK2}. First, we will prove a sharp upper bound for the volume of this join (Theorem \ref{TTTT}) which in particular implies the currently best known general upper bound on the mixed volume of the respective convex bodies (Corollary \ref{COREL}). Then we formulate the corresponding F\'ary--R\'edei type conjecture (Conjecture \ref{QQQ}) and prove that it implies the inequality part of Conjecture \ref{KK2}.

Recall that we keep the notation established above, in particular in Sections \ref{s:lowdim} and \ref{s:higher2}. Analogously to Section \ref{s:join1}, for any $p\in\NN$, $\bt\in\Theta_p$, and $K\in\calK(\RR^n)$ we set
$$
R_{p, \bt, K}:=-t_0\Delta_pK\vee\bigvee_{l=1}^pt_l\iota_lK.
$$
Then for $E:=\{0\}\times\mathbb{R}^{np}\subset\RR^p\times\RR^{np}$ we have $P_E(G_{p, \bm{t}, K})=R_{p, \bm{t}, K}$, see \eqref{eq:defGptK}. This observation will be used in the proof of the upper bound for $R_{p, \bt, K}$. Since no analog of Lemma \ref{lem:reduction} is known to us in this setting, the last part of the proof of the following theorem cannot be reduced to the case $p=1$ proven in \cite[\S5.4]{ArtsteinETAL15}; instead, we will use a different argument to show directly that if a simplex attains equality, it must have a vertex at the origin.

\begin{thm}\label{TTTT}
Let $p\in\mathbb{N}$. For any $K\in\mathcal{K}(\mathbb{R}^n)$ with $0\in K$ and any $\bt\in\Theta_p$, one has
\begin{align}
\label{eq:join2}
|R_{p,\bm{t}, K}|\leq\frac{(n!)^p}{(np)!}\sum_{\bm{j}\in\mathcal{J}^n_p}{n \choose \bj} \bt^{n-\bj}|K|^p.
\end{align}
If $\dim(K)=n$ and equality holds for some $\bt\in\relint(\Theta_p)$, then $K$ is a simplex with a vertex at the origin.
\end{thm}

\begin{proof}
    
Fix $\bt\in\Theta_p$ and $K\in\calK(\RR^n)$ with $0\in K$. Consider the subspace $E:=\RR^p\times\{0\}\subset\RR^p\times\RR^{np}$. Since $0\in K$, we have $ E\cap G_{p,\bt,K}=\hat\Theta_p$; in particular $|E\cap G_{p,\bt,K}|=\frac1{p!}$. Using this and $P_{E^\perp} G_{p,\bt,k}=R_{p,\bt,K}$, Lemmas \ref{L2} and \ref{L3a} imply
        \begin{align*}
                |R_{p,\bt, K}|\leq p!{np+p \choose p}|G_{p,\bt, K}|=\frac{(np+p)!}{(np)!}|G_{p,\bt, K}|\leq\frac{(n!)^p}{(np)!}\sum_{\bm{j}\in\mathcal{J}^n_p}{n \choose \bj} \bt^{n-\bj}|K|^p,
            \end{align*}
        proving \eqref{eq:join2}.

Assume that $\dim K=n$ and that equality holds in \eqref{eq:join2} for some $\bt\in\relint(\Theta_p)$. First, we in particular have equality in \eqref{eq:L3a}. Since \eqref{eq:L3a} was the only inequality used in the proof of Theorem \ref{T4a}, we in turn have equality in \eqref{kokl} and hence Theorem \ref{T4a} implies that $K$ is a simplex.

Second, equality in the application of Lemma~\ref{L2} implies that the bodies
$$
S_{\ba}:=\big\{\hat{s}\in\hat\Theta_p : (\hat\bs, \ba)\in G_{p,\bm{t},K}\big\},\quad\ba\in E^{\perp},
$$
are all homothetic to $G_{p,\bm{t},K}\cap E=\hat\Theta_p$. Let $K=\conv\{v_0, \dots, v_n\}$ and let $A_j(x)=\alpha_j(x)+\beta_j$  for $x\in\RR^n$ and $j=0,\dots,n$ be as in \S\ref{ss:simplices}, that is, $\sum_{j=0}^nA_j(x)=1$, $x\in\RR^n$, $A_j(v_i)=\delta_{ij}$, and  
\begin{align}
\label{eq:K=}
K=\{x\in\mathbb{R}^n: A_j(x)\geq 0,\, j=0, \dots, n\}.
\end{align}
Since $0\in K$, we have $\beta_j=A_j(0)\geq 0$ for $ j=0, \dots, n$. Moreover,  $\sum_{j=0}^n\beta_j=1$ and $\sum_{j=0}^n\alpha_j=0$. Observe also that $\alpha_1, \dots, \alpha_n$ are linearly independent.

We claim that
\begin{align}
\label{eq:Sa}
S_{\ba}=\left\{\hat\bs\in\hat\Theta_p\mid \sum_{j=0}^n m_j(\ba, \bm{s})\leq 0 \right\},
\end{align}
where
$$
m_j(\ba, \bm{s})\coloneqq\max\big\{-s_0t_0\beta_j, -\alpha_j(a_1)-s_1t_1\beta_j, \dots, -\alpha_j(a_p)-s_pt_p\beta_j\big\},\quad j=0,\dots,n.
$$
To prove this, observe first that $\hat\bs\in S_{\ba}$ if and only if $\ba=-s_0t_0\Delta_p(x)+\sum_{l=1}^p s_lt_l\iota_l(y_l)$ for some $x, y_1, \dots, y_p\in K$ or, equivalently, if and only if there exists
$$
u\in s_0t_0K\cap\bigcap_{l=1}^p( s_lt_lK-a_l).
$$
By \eqref{eq:K=}, this in turn is equivalent to the existence of $u\in\mathbb{R}^n$ satisfying 
\begin{align}
\label{eq:alphaju}
\alpha_j(u)\geq m_j(\ba, \bm{s}),\quad j=0, \dots, n.
\end{align}
One the one hand, because $\sum_{j=0}^n\alpha_j=0$, \eqref{eq:alphaju} implies $\sum_{j=0}^n m_j(\ba, \bm{s})\leq0$. 
Conversely, if $\sum_{j=0}^n m_j(\ba, \bm{s})\leq0$, then one can choose numbers
$r_0, \dots, r_n$ such that $r_j\geq m_j(\ba, \bm{s})$ for all $j$ and
$\sum_{j=0}^n r_j=0$. Since $\alpha_1, \dots, \alpha_n$ are linearly independent and $\sum_{j=0}^n \alpha_j=0$, there exists
$u\in\mathbb{R}^n$ such that $\alpha_j(u)=r_j\geq m_j(\ba, \bm{s})$ for every $j$. Hence $\bm{s}\in S_{\ba}$ and 
\eqref{eq:Sa} follows.

Assume that $0$ is not a vertex of $K$. Since $\sum_{j=0}^n\beta_j=1$, at least two of the numbers $\beta_0, \dots, \beta_n\geq0$ are strictly positive; in particular, there exists $i\in\{0,\dots,n\}$ with $\beta_i\in(0,1)$. Let $q,\varepsilon_2, \dots, \varepsilon_p>0$ be such that
$$q\sum_{l=0}^p t_l^{-1}<1\qquad\text{and}\qquad \sum_{l=2}^p\varepsilon_l=1-q\sum_{l=0}^p t_l^{-1}.$$
Consider $\bs^*\in\relint(\Theta_p)$ with
$$ s_0^*\coloneqq \frac{q}{t_0},\qquad s_1^*\coloneqq \frac{q}{t_1},\qquad\text{and}\qquad s_l^*\coloneqq \frac{q}{t_l}+\varepsilon_l,\quad l=2, \dots, p,
$$
and the numbers $x_j$, $j=0, \dots, n$, with $    x_i\coloneqq q$ and $x_j\coloneqq -\frac{q}{n}$, $j\neq i$. 
Then $x_0+\dots+x_n=0$ and therefore there exists $a_1\in\mathbb{R}^n$ such that $\alpha_j(a_1)=x_j$, $j=0, \dots, n$. Set
$$\ba^*\coloneqq(a_1, 0, \dots, 0)\in(\mathbb{R}^n)^p.$$
First, if $j=i$, we have
$$-\alpha_i(a_1)-s_1^*t_1\beta_i=-q-q\beta_i<-q\beta_i=-s_0^*t_0\beta_i,$$
and
$$-s_l^*t_l\beta_i=-(q+t_l\varepsilon_l)\beta_i<-q\beta_i=-s_0^*t_0\beta_i,\quad l=2,\dots,p.$$
Thus
$$m_i(\ba^*, \bm{s}^*)=-s_0^*t_0\beta_i.$$
Second, for $j\neq i$, we have
$$-\alpha_j(a_1)-s_1^*t_1\beta_j=\frac{q}{n}-q\beta_j>-q\beta_j=-s_0^*t_0\beta_j,$$
and
$$\big(-\alpha_j(a_1)-s_1^*t_1\beta_j\big)-\big(-s_l^*t_l\beta_j\big)=
\frac{q}{n}+t_l\varepsilon_l\beta_j>0,\quad l=2,\dots,p.$$
Hence, for $j\neq i$,
$$m_j(\ba^*, \bm{s}^*)=-\alpha_j(a_1)-s_1^*t_1\beta_j.$$
Since all of the above inequalities are strict, there exists a relatively open neighborhood $U\subset\relint(\Theta_p)$ of $\bm{s}^*$ such that for all $\bm{s}\in U$ we have
$$\sum_{j=0}^n m_j(\ba^*, \bm{s})=-s_0t_0\beta_i+\sum_{j\neq i}\big(-\alpha_j(a_1)-s_1t_1\beta_j\big).$$
Moreover, since $\sum_{j=0}^n m_j(\ba^*, \bm{s}^*)=0,$ we have $\hat\bs^*\in\partial S_{\ba^*}$. In other words, the relative boundary of $S_{\ba^*}$ is locally given by a single affine equation. Using \eqref{eq:Sa}, it is easily verified that  the corresponding normal vector is not a multiple of any outer facet normal of $\hat\Theta_p$. Therefore $S_{\ba^*}$ cannot be homothetic to $\hat\Theta_p$, contradicting the necessary consequence of the equality condition in Lemma~\ref{L2}, and we conclude that $0$ is a vertex of $K$.
    \end{proof}

\begin{cor}\label{COREL}
Let  $p\in\NN$ and $\bm{j}\in\mathcal{J}^n_p$ with $0<j_0, \dots, j_p<n$. For any  $K\in\mathcal{K}(\mathbb{R}^n)$ one has
$$   W_{p,K}(\bm{j})\leq \frac{(n!)^p}{(np)!}\frac{n^n}{j_0^{j_0}\cdots j_p^{j_p}}|K|^p.$$
\end{cor}

\begin{proof}
We may assume $0\in K$. Let $\bt\in\relint(\Theta_p)$. Theorem \ref{TTTT} together with the monotonicity of mixed volume applied to $-t_0\Delta_pK\subset R_{p,\bm{t}, K}$ and $t_l\iota_lK\subset R_{p, \bm{t}, K}$, $l=1,\dots,p$, implies
        \begin{align*}
                W_{p, K}(\bm{j})\leq (\bt^{n-\bj})^{-1}|R_{p, \bm{t}, K}|\leq \frac{(n!)^p}{(np)!}\bt^{\bj}\big(\sum_{l=0}^pt_l^{-1}\big)^{n}|K|^p,
            \end{align*}
 Choosing $\bt\in\relint(\Theta_p)$ subject to $j_it_i\sum_{l=0}^p t_l^{-1}=n$, $i=0,\dots,p$, then easily yields the result.
\end{proof}

To finish, let us propose a F\'ary--R\'edei type conjecture analogous to Conjecture \ref{QQ}. Observe that also the following conjecture subsumes \cite[Conjecture 1.2]{ArtsteinETAL15} for $p=1$.

\begin{con}\label{QQQ}
    Let $p\in\mathbb{N}$. For any $K\in\mathcal{K}(\mathbb{R}^n)$ with $\dim(K)=n$, and any $\bm{t}\in\Theta_p$, there exists $x\in K$ such that
    $$\frac{|R_{p,\bm{t}, K-x}|}{|K|^p}\leq\frac{|R_{p,\bm{t}, S_n}|}{|S_n|^p},$$
    where $S_n\in\mathcal{K}(\mathbb{R}^n)$ is a centered $n$-simplex. 
\end{con}

In order to prove that Conjecture \ref{QQQ} implies the inequality part of Conjecture \ref{KK2}, we will need the following lemma.

\begin{lemma}\label{intern}
Let $p\in\mathbb{N}$ and $\bm{j}\in\mathcal{J}^n_p$, and let $S_n\in\mathcal{K}(\mathbb{R}^n)$ be a
centered $n$-simplex. There exists $\bt\in\relint(\Theta_p)$ such that
\begin{align}
\label{eq:intern}
|R_{p,\bm{t},S_n}|\leq\frac{(n!)^p}{(np)!}{n \choose \bm{j}}\bm{t}^{n-\bm{j}}|S_n|^p.
\end{align}
\end{lemma}

\begin{proof}For the sake of brevity, we will denote $\calL:=\{0, \dots, p\}$ and $\calI:=\{0, \dots, n\}$. Without loss of generality, we may assume that $S_n=\conv\{v_0, \dots, v_n\},$ where \begin{align*}
    v_0\:= -(e_1+\dots+e_n)\qquad\text{and}\qquad v_i:= e_i, \quad i\in\mathcal{I}\setminus\{0\}.
\end{align*} 
Then, for any $\bt\in\Theta_p$, we have $R_{p,\bm{t}, S_n}=\conv\big\{\bu^l_{i} : (l, i)\in\mathcal{L}\times\mathcal{I}\big\}$, where for $i\in\mathcal{I}$ we set
\begin{align*}\bu^0_i\coloneqq -t_0\Delta_p(v_i) \qquad\text{and}\qquad \bu^l_i\coloneqq t_l\iota_l(v_i),\quad l\in\mathcal{L}\setminus\{0\}.\end{align*}
Observe that, in particular, $R_{p,\bt, S_n}$ is a polytope. Since $0\in\inter(S_n)$, we have $0\in\inter(R_{p,\bm{t},S_n}).$ Consequently, we can express the volume of $R_{p,\bm{t}, S_n}$ as follows:
\begin{align}
\label{eq:volR}
|R_{p,\bm{t}, S_n}|=\sum_{F\in\calF}\big|\conv(\{0\}\cup F)\big|,
\end{align}
where $\calF$ denotes the set of facets of $R_{p,\bm{t}, S_n}$.

Fix a facet $F\in\mathcal{F}$. We have $F=\{\bm{y}\in R_{p,\,\bm{t},\,S_n}: \bm{y}\cdot \bm{x}=1\}$ for some $\bm{x}\in R_{p,\bm{t}, S_n}^{\circ}$. For all $i\in\calI$ and $l\in\calL$ we have $\bx\cdot\bu_i^l\leq 1$. Let us denote for $i\in\calI$
$$
z_i^0:=1-\bx\cdot\bu_i^0=1+t_0\sum_{l=1}^px_l\cdot v_i\qquad\text{and}\qquad z_i^l:=1-\bx\cdot\bu_i^l=1-t_lx_l\cdot v_i,\quad l\in\calL\setminus\{0\}.
$$
Then $z_i^l\geq0$ for all $i\in\calI$ and $l\in\calL$, and $z_i^l=0$ if and only if $\bu_i^l\in F$.

Consider $\bt=(t_0,\dots, t_p)\in\relint(\Theta_p)$ given by
$$t_l:=\left(\sum_{m=0}^p\frac{(p+1)j_l+1}{(p+1)j_m+1}\right)^{-1}, \quad l\in \calL,$$
and denote also
$$r_l:=\frac{j_l+(p+1)^{-1}}{n+1}, \qquad l\in\mathcal{L}.$$
Observe that $r_l\in(0,1)$.  Using that $S_n$ is centered, that is $\sum_{i=0}^nv_i=0$, we get
\begin{align}
\label{eq:sum_i}
\sum_{i=0}^nr_lz_i^l=j_l+(p+1)^{-1},\quad l\in\calL.
\end{align}
It can be verified by a direct computation that for $i\in\calI$ one has $\sum_{l=0}^pr_lz_i^l=1$. In particular, $r_lz_i^l\in[0,1]$, and hence there are at least $j_l+1$ non-zero summands in \eqref{eq:sum_i} for every $l\in\calL$.
In other words, for every $l\in\calL$, there are at most $n-j_l$ indices $i\in\calL$ such that $z_i^l=0$ or, equivalently, $\bu_i^l\in F$. Denote $\calE_F:=\{(l,i)\in\calL\times\calI\mid z_i^l=0\}$. Then $F=\conv\{\bu_i^l\mid (l,i)\in\calE_F\}$. Since $\dim F=np-1$, we have
$$
np\leq \#\calE_F\leq \sum_{l=0}^p(n-j_l)=np.
$$
Consequently, $\#\big(\mathcal{E}_F\cap(\{l\}\times\mathcal{I})\big)= n-j_l$ for any $l\in\calL$, and $F$ is an $(np-1)$-simplex. Thus,
\begin{align}
\label{eq:vol0F}
\big|\conv(\{0\}\cup F)\big|=\frac{1}{(np)!}\big|\det\big(\bu^l_i : (l, i)\in\mathcal{E}_F\big)\big|.
\end{align}
Set $\bw^0_i:=-\Delta_p(v_i)$ and $\bw^l_i:=\iota_l(v_i)$ so that $\bu^l_i=t_l\bw^l_i$ for $ (l, i)\in\calL\times\calI$. For any $\calE\subset\calL\times\calI$ with $\#\calE=np$, consider the matrix $W_{\calE}:=\big(\bw_i^l\mid (l,i)\in\calE\big)\in\ZZ^{(np)\times (np)}$. Observe that $|\det W_\calE|\geq1$; in particular, $|\det W_\calE|^2\geq |\det W_\calE|$. Further, set $$\widetilde\calF:=\Big\{\mathcal{E}\subset\mathcal{L}\times\mathcal{I}: \#\big(\mathcal{E}\cap(\{l\}\times\mathcal{I})\big)=n-j_l\text{ for all } l\in\calL\Big\}.$$
Altogether, using \eqref{eq:volR} and \eqref{eq:vol0F}, we can write
$$
|R_{p,\bt,S_n}|=\frac{\bt^{n-\bj}}{(np)!}\sum_{F\in\calF}\big|\det(W_{\calE_F})\big|\leq\frac{\bt^{n-\bj}}{(np)!}\sum_{\calE\in\widetilde\calF}\big|\det(W_{\calE})\big|^2=\frac{(n!)^p}{(np)!}{n \choose \bm{j}}\bt^{n-\bj}|S_n|^p,
$$
where in the last step we used $|S_n|=\frac{n+1}{n!}$, and the Cauchy--Binet formula to compute
$$
\sum_{\calE\in\widetilde\calF}|\det(W_{\calE})|^2=(n+1)^p{n \choose \bj}.
$$
This shows \eqref{eq:intern}.
\end{proof}

\begin{rem}
We note without proof that one in fact has equality in \eqref{eq:intern} for all $\bt\in\relint(\Theta_p)$ satisfying $j_l<(n+1)r_l(\bt)<j_l+1$ for $l=0, \dots, p$, where we denote
$$
r_l(\bm{t}):=\left(\sum_{m=0}^pt_m^{-1}\right)^{-1}t_l^{-1}.
$$  
This generalizes \cite[Lemma 2.1]{ArtsteinETAL15}.
\end{rem}

\begin{prop}
\label{posledni}
    Conjecture \ref{QQQ} implies inequality \eqref{KOTRB} in Conjecture \ref{KK2}.
\end{prop}

    \begin{proof}
Let $p\in\NN$, $\bm{j}\in\mathcal{J}_p^n$, and assume that $K\in\calK(\RR^n)$ with $\dim K=n$ satisfies Conjecture \ref{QQQ}. Fix a centered $n$-simplex $S_n\in\calK(\RR^n)$. Choose $\bt\in\relint(\Theta_p)$ as in Lemma \ref{intern} and $x\in K$ such that
$$
\frac{|R_{p,\bm{t}, K-x}|}{|K|^p}\leq\frac{|R_{p,\bm{t}, S_n}|}{|S_n|^p}.
$$
Then Lemma \ref{intern} together with the translation invariance of mixed volumes and the homogeneity and monotonicity applied to inclusions
$$
-t_0\Delta_p(K-x)\subset R_{p, \bt, K-x}\qquad\text{and}\qquad t_l\iota_l(K-x)\subset R_{p, \bt,K-x},\quad l=1, \dots, p,
$$
yields
$$
               W_{p, K}(\bm{j})\leq \bm{t}^{\bm{j}-n}|R_{p, \bm{t}, K-x}|\leq\frac{(n!)^p}{(np)!}{n \choose \bm{j}}|K|^p,
$$
which completes the proof.

    \end{proof}

\bibliographystyle{abbrv}
\bibliography{ref}

\end{document}